\documentclass[reqno,a4paper]{amsart}
\usepackage[pagebackref]{hyperref}
\usepackage{enumerate}

\allowdisplaybreaks

\newtheorem{theorem}{Theorem}[section]
\newtheorem{lemma}[theorem]{Lemma}
\newtheorem{corollary}[theorem]{Corollary} 
\newtheorem{proposition}[theorem]{Proposition} 
\newtheorem{question}[theorem]{Question} 

\theoremstyle{definition}
\newtheorem{definition}[theorem]{Definition}
\newtheorem{example}[theorem]{Example} 

\theoremstyle{remark}
\newtheorem{remark}[theorem]{Remark}

\numberwithin{equation}{section}

\newcommand{\ind}[1]{1_{#1}} 
\newcommand{\indbr}[1]{1_{\{#1\}}} 
\newcommand{\EE}{\mathbb{E}} 
\newcommand{\PP}{\mathbb{P}} 
\newcommand{\RR}{\mathbb{R}} 
\newcommand{\eps}{\varepsilon}
\DeclareMathOperator{\Ent}{Ent}	
\DeclareMathOperator{\Med}{Med}	
\DeclareMathOperator{\Var}{Var}	
\DeclareMathOperator{\cl}{cl} 	

\newcommand{\cT}{\mathcal{T}}
\newcommand*{\cTbar}{\overline{\mathcal{T}}}
\newcommand*{\Tbar}{\overline{\mathbf{T}}\vphantom{\mathbf{T}}}

\newcommand*{\dualnorm}[1][p]{_{\theta, #1}}

\begin{document}

\title[Convex Poincar\'e inequality and weak transportation inequalities]{On the convex Poincar\'e inequality and weak transportation inequalities}

\author{Rados{\l}aw Adamczak}
\address{Institute of Mathematics, University of Warsaw, Banacha 2, 02--097 Warsaw, Poland.}
\email{R.Adamczak@mimuw.edu.pl}

\thanks{Research partially supported by the National Science Centre, Poland, grants
no. 2015/18/E/ST1/00214 (R.A.)
and 2015/19/N/ST1/00891 (M.St.).}

\author{Micha{\l} Strzelecki}
\address{Institute of Mathematics, University of Warsaw, Banacha 2, 02--097 Warsaw, Poland.}
\email{M.Strzelecki@mimuw.edu.pl}

\date{Last changes : March 27, 2017.}

\begin{abstract}

We prove that for a probability measure on $\RR^n$, the Poincar\'e inequality for convex functions is equivalent to the weak transportation inequality with a quadratic-linear cost. This generalizes recent results by Gozlan et al. and Feldheim et al., concerning probability measures on the real line.

The proof relies on modified logarithmic Sobolev inequalities of Bobkov-Ledoux type for convex and concave functions, which are of independent interest.

We also present refined concentration inequalities for general (not necessarily Lipschitz) convex functions, complementing recent results by Bobkov, Nayar, and Tetali.
\end{abstract}

\maketitle

\section{Introduction} In the last thirty years a substantial body of research has been devoted to the interplay between various functional inequalities, transportation of measure theory, and the concentration of measure phenomenon, showing intimate connection between them. While most of the investigations have been carried out in the setting of general Lipschitz  functions, concentration inequalities restricted to the class of convex Lipschitz functions have also been considered by many authors, starting from the seminal work by Talagrand in the 1990's (\cite{MR1122615,TalConcMeasure}, see also \cite{MR1399224,MR1097258,MR1756011,MR2073425} and the monograph \cite{MR1849347} for subsequent developments). A crucial feature of these results is that they are satisfied under much less restrictive assumptions concerning the regularity of the underlying probability measure when compared to inequalities valid for all Lipschitz functions. Even though the theory of concentration of measure for convex functions to some extent parallels the classical theory, there are some subtle differences related to the fact that convexity is not preserved under general contractions---even under the change of signs---which creates certain difficulties in the proofs and makes many well known arguments, which have been established in the classical context, invalid. As a consequence, the theory of concentration of measure for convex functions has not yet reached a satisfactory level of completeness. Nevertheless, several important results have been obtained in recent years, connecting dimension-free concentration inequalities for convex functions with the convex Poincar\'e inequality \cite{MR3311918} and a new type of weak transportation cost inequalities \cite{gozlan,gozlan_new}. We will now briefly describe these developments, which will allow us to formulate our main result.

Let $|\cdot|$ stand for the standard Euclidean norm on $\RR^n$. Let $\mu$ be a Borel probability measure on $\RR^n$ and let $X$ be a random vector with law $\mu$.
We say that $\mu$ (equivalently $X$) satisfies the convex Poincar\'e inequality with constant $\lambda > 0$ if for all convex functions $f \colon \RR^n \to \RR$ we have
\begin{align}\label{eq:convex-Poincare}
\Var f(X) \le \frac{1}{\lambda} \EE |\nabla f(X)|^2,
\end{align}
where by $|\nabla f(x)|$ we mean the length of gradient at $x$, defined as
\begin{equation}\label{eq:grad-length}
|\nabla f(x)| = \limsup_{y \to x} \frac{|f(y) - f(x)|}{|y-x|}.
\end{equation}
Note that this coincides with the length of the `true' gradient provided $f$ is differentiable at $x$.
Also, it is enough to assume that \eqref{eq:convex-Poincare} holds for convex Lipschitz functions, since an arbitrary convex function can be pointwise approximated by convex Lipschitz functions.

It follows from the results by Gozlan, Roberto, and Samson  \cite{MR3311918} that $\mu$ satisfies the convex Poincar\'e inequality if and only if there exists a constant $c>0$ such that for any $N$, any convex set $A \subseteq (\RR^n)^N$ with $\mu^{\otimes N}(A) \ge 1/2$, and any $t > 0$,
\begin{align}\label{eq:enlargement}
\mu^{\otimes N}(A+t B_2^{Nn}) \ge 1 - 2 \exp(-ct),
\end{align}
where $B_2^k$ denotes the unit Euclidean ball in $\RR^k$ and $+$ stands for the Minkowski addition.

It is not difficult to see that \eqref{eq:enlargement} is equivalent to the one-sided deviation inequality for convex $1$-Lipschitz functions, i.e.
\begin{align}\label{eq:upper-tail-introduction}
\PP( f(X_1,\ldots,X_N) \ge \Med f(X_1,\ldots,X_N)  + t) \le 2e^{-ct}
\end{align}
for all $t \ge 0$, where $X_1,\ldots,X_N$ are i.i.d.\@ copies of $X$, and $\Med Y$ denotes the median of the random variable $Y$, i.e. $\Med Y = \inf\{t\in \RR\colon \PP(Y \le t) \ge 1/2\}$.

Thus the convex Poincar\'e inequality is equivalent to a dimension free deviation inequality for the upper tail of convex Lipschitz functions.

Let us now pass to the connections between the Poincar\'e inequality and transportation inequalities. Let $\theta \colon \RR^n \to [0,\infty]$ be a measurable function with $\theta(0) = 0$. Recall that the optimal transport cost between two probability measures $\mu$ and $\nu$ on $\RR^n$, induced by $\theta$ is given by
\begin{align}\label{eq:standard-cost}
\cT_\theta(\nu,\mu) = \inf_{\pi}\int_{\RR^n}\int_{\RR^n}\theta(x - y)\pi(dxdy),
\end{align}
where the infimum is taken over all couplings between $\mu$ and $\nu$, i.e. over all probability measures on $(\RR^{n})^2$ such that $\pi(dx \times \RR^n) = \mu(dx)$, $\pi(\RR^n\times dy) = \nu(dy)$. Recall also that the relative entropy $H(\nu|\mu)$ is defined as
\begin{align}\label{eq;relative-entropy}
H(\nu|\mu) = \int_{\RR^n} \log \frac{d\nu}{d\mu}d\nu,
\end{align}
if $\nu$ is absolutely continuous with respect to $\mu$ and $H(\nu|\mu) = \infty$ otherwise.

It has been  proved in \cite{MR1846020} that $\mu$ satisfies the Poincar\'e inequality \eqref{eq:convex-Poincare} for \emph{all smooth} functions if and only if there exist constants $C,D$ such that for all probability measures $\nu$,
\begin{align}\label{eq:classical-transportation}
\cT_{\theta_{C,D}}(\nu,\mu) \le H(\nu|\mu),
\end{align}
where
\begin{align}\label{eq:cost-function}
\theta_{C,D}(x) =
\begin{cases}
\frac{|x|^2}{2C} & \text{for } |x| \le CD,\\
D|x| - \frac{CD^2}{2} & \text{for } |x| > CD.
\end{cases}
\end{align}

	Recently Gozlan, Roberto, Samson, Shu, and Tetali \cite{gozlan_new} formulated a similar characterization of the \emph{convex}  Poincar\'e inequality \emph{on the real line}. In order to formulate their result we need to introduce the weak transport cost between probability measures and corresponding transportation inequalities as defined in \cite{gozlan,gozlan_new}.

In what follows, by $\mathcal{P}_1(\RR^n)$ we denote the class of all probability measures $\nu$ on $\RR^n$ such that $\int_{\RR^n} |x|d\nu(x) < \infty$.

\begin{definition}
Let $\mu$ and $\nu$ be probability measures on $\RR^n$. Assume that $\nu \in \mathcal{P}_1(\RR^n)$. For a convex, lower semicontinuous function $\theta\colon \RR^n \to [0,\infty]$, such that $\theta(0) = 0$ define the weak transport cost between $\mu$ and $\nu$ as
\begin{displaymath}
\cTbar_{\theta}(\nu|\mu) = \inf_{\pi}\int_{\RR^n}\theta\bigl(x - \int_{\RR^n} y p_x(dy)\bigr)\mu(dx),
\end{displaymath}
where the infimum is taken over all couplings between $\mu$ and $\nu$ and for $x \in \RR^n$, $p_x(\cdot)$ is the conditional measure defined ($\mu$ almost surely) by $\pi(dxdy) = p_x(dy)\mu(dx)$.
\end{definition}

Note that in the probabilistic notation one can write
\begin{displaymath}
\cTbar_{\theta}(\nu|\mu) = \inf_{(X,Y)} \EE \theta(X - \EE(Y|X)),
\end{displaymath}
where the infimum is taken over all pairs of random vectors $(X,Y)$ with values in $\RR^n \times \RR^n$, such that $X$ is distributed according to $\mu$ and $Y$ according to $\nu$.

Due to the asymmetry between $\mu$ and $\nu$, one can now introduce three different inequalities related to the cost $\cTbar_{\theta}$.
\begin{definition} Let $\mu \in \mathcal{P}_1(\RR^n)$ and $\theta\colon \RR^n \to [0,\infty]$ be a~convex lower semicontinuous function with $\theta(0) = 0$. We will say that $\mu$ satisfies the inequality
\begin{itemize}
\item $\Tbar_{\theta}^+$ if for every probability measure $\nu\in \mathcal{P}_1(\RR^n)$,
\begin{displaymath}
\cTbar_{\theta}(\nu|\mu) \le H(\nu|\mu),
\end{displaymath}
\item $\Tbar_{\theta}^-$ if for every probability measure $\nu\in \mathcal{P}_1(\RR^n)$,
\begin{displaymath}
\cTbar_{\theta}(\mu|\nu) \le H(\nu|\mu),
\end{displaymath}
\item $\Tbar_{\theta}$ if $\mu$ satisfies both $\Tbar_{\theta}^+$ and $\Tbar_{\theta}^-$.
\end{itemize}
\end{definition}

The definition of those inequalities in \cite{gozlan} differs formally from the one presented above (which is taken from \cite{gozlan_new}). It is not difficult to see that the definitions presented in both articles are equivalent up to universal constants---the version above is more convenient for our purposes.

The authors of \cite{gozlan_new} proved that a probability measure $\mu$ on the real line satisfies the convex Poincar\'e inequality for some constant $\lambda > 0$ if and only if it satisfies the transportation inequality $\Tbar_{\theta_{C,D}}$ for some $C,D > 0$. In a dual formulation (expressed in terms of infimum convolution inequalities), this result has been also obtained in \cite{Feldheim}.

Our main result is an extension of this equivalence to arbitrary dimension.

\begin{theorem} \label{thm:main} Let $\mu$ be a probability measure on $\RR^n$. Then the following conditions are equivalent:
\begin{itemize}
\item[(i)] There exists $\lambda > 0$ such that $\mu$ satisfies the convex Poincar\'e inequality \eqref{eq:convex-Poincare}.
\item[(ii)] There exist $C,D > 0$ such that $\mu$ satisfies the transportation inequality $\Tbar_{\theta_{C,D}}$.
\end{itemize}
\end{theorem}

\begin{remark}\label{rem:dependence-of-constants}
The implication (ii) $\implies$ (i) is standard, in this case
$\lambda = \frac{1}{C}$.
 In our proof the constants $C,D$ in the implication (i) $\implies$ (ii) depend not only on $\lambda$ but also on certain quantiles related to the measure $\mu$ (which are always finite but may be of the order of up to $\sqrt{n}$). This is related to the inequality $\Tbar_{\theta_{C,D}}^+$ responsible for the lower tail of convex functions, which is usually more difficult to deal with than the upper tail. We suspect that this is an artefact of our proof and one should be able to obtain $\Tbar^+_{\theta_{C,D}}$ with $C,D$ depending only on $\lambda$. As for $\Tbar_{\theta_{C,D}}^-$ our argument does yield it with  $C,D$ depending only on $\lambda$ (see Corollary~\ref{cor:Poincare-transport-minus} below for details).
\end{remark}

\begin{remark}\label{re:1.5}
Thanks to well known tensorization properties of the inequality $\Tbar_{\theta_{C,D}}$, Theorem \ref{thm:main} implies that the convex Poincar\'e inequality is equivalent to improved two-level dimension free concentration inequality for convex functions (see Example~\ref{ex:norms-and-concentration-2} below for a precise formulation). In the class of Lipschitz functions inequalities of this type have been first obtained by Talagrand \cite{MR1122615} in the case of the product exponential distribution (with an alternate proof, using infimum-convolution inequalities, by Maurey \cite{MR1097258}). The fact that they are consequences of the Poincar\'e inequality for smooth functions was established by Bobkov and Ledoux \cite{MR1440138}. By results due to Gozlan et al. \cite{MR3311918} this can be regarded as a self-improvement of dimension-free concentration properties of Lipschitz functions. Our result shows that similar self-improvements are present also in the setting of convex concentration.
\end{remark}

\begin{remark}
 In \cite{bobkov-goetze} Bobkov and G\"otze provide a simple characterization of measures on $\RR$ which satisfy the convex Poincar\'e inequality for some $\lambda > 0$ (and thus also the inequality $\Tbar_{C,D}$) in terms of the probability distribution function. A similar characterization for larger $n$ seems to be a non-trivial open problem.
\end{remark}

The organization of the article is as follows. First, in Section \ref{sec:preliminaries}, we present preliminary properties of measures satisfying the convex Poincar\'e inequality and weak transportation inequalities, to be used in the proofs. Section \ref{sec:log-Sobolev} contains our most important technical result, i.e. modified log-Sobolev inequalities for convex and concave functions, which in Section \ref{sec:main-proof} are combined with the Hamilton-Jacobi equations giving the proof of Theorem \ref{thm:main}.

Next, in Section \ref{sec:Examples} we briefly discuss operations preserving the convex Poincar\'e inequality, which may be used to provide new non-trivial examples of measures satisfying it.

In Section \ref{sec:concentration} we present refined concentration of measure inequalities, which are consequences of weak transportation inequalities. We consider there more general cost functions than the one corresponding to the convex Poincar\'e inequality and discuss applications both to the Lipschitz and non-Lipschitz setting.

Finally, in Section \ref{sec:final} we state a few open questions. The Appendix contains basic facts concerning Hamilton-Jacobi equations, which are used in the proof of Theorem \ref{thm:main}.

\section{Preliminaries on the convex Poincar\'e inequality and weak transportation inequalities}\label{sec:preliminaries}

In this section we present basic concentration of measure properties implied by the  convex Poincar\'e inequality and the dual formulations of  weak transportation inequalities. They will be needed in the proof of our main result.

We begin with a simple reformulation of the convex Poincar\'e inequality.

\begin{lemma}\label{le:Poincare-median}
Let $X$ be a random vector in $\RR^n$ satisfying the convex Poincar\'e inequality \eqref{eq:convex-Poincare}. Then for every convex function $f \colon \RR^n \to \RR$,
\begin{equation*}
\EE (f(X)-\Med f(X))^2 \leq \frac{2}{\lambda} \EE |\nabla f(X)|^2.
\end{equation*}
\end{lemma}

\begin{proof}
Note that for every random variable $Z$, thanks to the fact that the median minimizes the mean absolute deviation, we have
\begin{equation*}
 (\EE Z-\Med Z)^2\leq (\EE |Z-\Med Z|)^2 \leq (\EE |Z-\EE Z|)^2\leq \Var Z.
\end{equation*} Thus
\begin{displaymath}
\EE (Z - \Med Z)^2 = \Var Z + (\EE Z - \Med Z)^2 \le 2 \Var Z
\end{displaymath}
and it is enough to set $Z = f(X)$ and apply \eqref{eq:convex-Poincare}.
\end{proof}

\subsection{Concentration inequalities}

Let us start with the already mentioned (see \eqref{eq:upper-tail-introduction}) upper tail estimate for convex Lipschitz functions implied by the convex Poincar\'e inequality. The proposition below can be also obtained by abstract results from \cite{MR3311918}, but we would like to provide an alternative derivation based on moments (the possibility of such a proof was suggested in \cite{MR3311918}). Our strategy mimics a well known approach from the general Lipschitz case (see e.g. Proposition 2.5. in \cite{MR2507637}), however we have to deal with some small difficulties related to the fact that in the convex setting we cannot truncate the function as this operation does not preserve convexity.

\begin{proposition}\label{prop:Poincare-upper-tail}
Assume that $X$ is a random vector in $\RR^n$, satisfying the convex Poincar\'e inequality \eqref{eq:convex-Poincare}. Then for any $L$-Lipschitz convex function $f\colon \RR^n \to \RR$ and any $t > 0$,
\begin{displaymath}
\PP(f(X) \ge \Med f(X) + t) \le 8e^{-0.52 \sqrt{\lambda}t/L}.
\end{displaymath}
\end{proposition}

\begin{proof}[Proof of Proposition~\ref{prop:Poincare-upper-tail}]
Consider the random variable $Y = (|X| - a)_+$, where $a \in \RR_+$ is arbitrary such that $\PP(|X| \le a) > 1/4$, and let $Y'$ be an independent copy of $Y$. Since the function $\varphi(x)=(|x|-a)_+$ is convex,
\begin{align*}
\frac{1}{\lambda}\PP(|X| \ge a) &= \frac{1}{\lambda}\EE |\nabla \varphi(X)|^2 \ge \Var Y = \frac{1}{2}\EE (Y - Y')^2 \\
&\geq \frac{1}{2}\EE (Y - Y')^2(\indbr{Y>0}\indbr{Y'=0} +\indbr{Y=0} \indbr{Y'>0})\\
&\ge \frac{1}{4}\EE Y^2\indbr{Y > 0} \ge \frac{2}{\lambda} \PP(|X| > a + 2\sqrt{2/\lambda})
\end{align*}
and so $\PP(|X| \ge a + 2\sqrt{2/\lambda}) \le 2^{-1}\PP(|X| \ge a)$, which implies that $|X|$ is exponentially integrable. In particular for every Lipschitz function $f$ and all $p > 0$, $\EE |f(X)|^p < \infty$.

Assume now that $f\colon\RR^n\to\RR$ is convex. Then for all $p \ge 2$, applying Lemma~\ref{le:Poincare-median} to the convex function $x\mapsto (f(x) - \Med f(X))_+^{p/2}$ (note that its median is zero and $|\nabla (f(x)-\Med f(X))_+|\le |\nabla f(x)|$), we obtain
\begin{align*}
\EE (f(X)-\Med f(X))_+^p &\le \frac{2}{\lambda} \cdot \frac{p^2}{4}\EE (f(X)-\Med f(X))_+^{p-2}|\nabla f(X)|^2 \\
&\le \frac{p^2}{2\lambda} \bigl(\EE (f(X)-\Med f(X))_+^p\bigr)^{1 - 2/p} \bigl(\EE |\nabla f(X)|^p\bigr)^{2/p},
\end{align*}
where we used H\"older's inequality with exponents $p/(p-2)$, $p/2$. If we additionally assume that $f$ is Lipschitz, so that $\EE (f(X)-\Med f(X))_+^p < \infty$, we get
\begin{align}\label{eq:Poincare-moments}
\bigl(\EE (f(X)-\Med f(X))_+^p\bigr)^{1/p} \le \frac{p}{\sqrt{2\lambda}} \bigl(\EE |\nabla f(X)|^p\bigr)^{1/p},
\end{align}
which via Chebyshev's inequality in $L_p$ implies
\begin{align}\label{eq:Poincare-moments-tail}
\PP \Big( f(X) \ge \Med f(X) + e\frac{p}{\sqrt{2\lambda}} \bigl(\EE |\nabla f(X)|^p\bigr)^{1/p}\Big) \le e^{2-p}
\end{align}
for $p \ge 0$.
Now, if the Lipschitz constant of $f$ equals one, the above inequality yields for $t > 0$,
\begin{equation*}
\PP(f(X) \ge \Med f(X) + t) \le \exp\Bigl(2 - \frac{\sqrt{2\lambda}}{e} t\Bigr) \le 8e^{-0.52 \sqrt{\lambda} t}. \qedhere
\end{equation*}
\end{proof}

\begin{remark}
Another possible approach is based on the Laplace transform: assume without loss of generality that $\EE f(X)=0$ and denote $M(s) = \EE e^{sf(X)}$ for $s\geq 0$. Since the function $e^{sf(\cdot)/2}$ is convex, the Poincar\'e inequality yields
\begin{equation*}
M(s) - M(s/2)^2 = \Var(e^{sf(X)/2}) \leq \frac{1}{4\lambda} \EE s^2 |\nabla f(X)|^2 e^{sf(X)} \leq \frac{L^2 s^2}{4\lambda} M(s).
 \end{equation*}
 The idea would be now to regroup the expressions appearing in the above inequality, repeat the procedure (with $s/2$ instead of $s$), and---after a simple limit argument---obtain a bound on $\EE e^{sf(X)}$. After that we could use Markov's inequality and optimize in $s$ to obtain an estimate of the upper tail of $f$. However a delicate issue emerges: we have to a priori know that  (for reasonable choices of the parameter $s$) $e^{sf(X)}$ is integrable (in the setting of smooth functions one overcomes this problem simply by truncating $f$, for convex functions one would need e.g. to repeat the beginning of the proof of Proposition \ref{prop:Poincare-upper-tail}); cf. the remark following Theorem 6.8 in~\cite{MR3311918}.
\end{remark}

We do not know if the convex Poincar\'e inequality implies similar tail estimates---which depend only on $\lambda$ and the Lipschitz constant of the function---for the lower tail of convex Lipschitz functions, i.e. for $\PP(f(X) \le \Med f(X) - t)$, $t>0$ (cf. Question ~\ref{q:lower-tail} below).

Nonetheless, we can easily get estimates in terms of $\lambda$ and certain quantiles of $X$. They will be crucial in the proof of the implication
\begin{gather*}
\text{Convex Poincar\'e inequality}  \implies \Tbar_{\theta_{C,D}}^+.
\end{gather*}

\begin{lemma}\label{le:conc-Med-concave-n}
Let $X$ be a random vector in $\RR^n$ satisfying the convex Poincar\'e inequality \eqref{eq:convex-Poincare} and let $M$ be any number such that  $\PP(|X - \EE X|\le M) \ge 3/4$.  Then for every convex $f\colon \RR^n \to \RR$  and for any $t > 32 M \EE|\nabla f(X)|$,
\begin{equation*}
\PP(f(X)\leq \Med f(X) - t) \leq 8e^{-t\sqrt{\lambda}/(32 \EE|\nabla f(X)|)}.
\end{equation*}
\end{lemma}

\begin{proof}
By Proposition~\ref{prop:Poincare-upper-tail} (note that the function $x\mapsto |x - \EE X | $ is convex and $1$-Lipschitz),
\begin{equation}\label{eq:EX}
\PP( |X - \EE X | \geq M + t) \leq 8e^{-0.5 t\sqrt{\lambda}}, \quad t\geq 0.
\end{equation}

Let $f\colon\RR^n\to\RR$ be a convex function. Without loss of generality we may assume $\Med f(X) =0$. We have $\PP(f(X)\geq 0) \geq 1/2$,
\begin{gather*}
\PP(|X-\EE X| \le M ) \geq 3/4,\\
\PP( |\nabla f(X)| < 8\EE|\nabla f(X)|)  \geq 7/8.
\end{gather*}
Thus there exists $x_0$ such that $f(x_0) \geq 0$, $|x_0 - \EE X| \le M$, and $|\nabla f(x_0)| <  8\EE|\nabla f(X)|$.
Define
\begin{equation*}
\widetilde{f} (x) = f(x_0) + \langle u,  x-x_0\rangle, \quad x\in\RR^n,
\end{equation*}
where $u$ is any subgradient of $f$ at $x_0$, so that $\tilde{f}(x) \le f(x)$ for all $x \in \RR^n$. Taking $x = x_0 + \varepsilon u$ with $\varepsilon \to 0$ we see that $|u| \le |\nabla f(x_0)| \le 8 \EE |\nabla f(X)|$, and thus we have
\begin{align*}
\PP(f(X)\leq -t)
&\leq \PP(\widetilde{f}(X) \leq -t) \leq \PP(\langle u,  X-x_0\rangle  \leq -t)\\
&\leq \PP(|u| |X - x_0|\geq t) \leq \PP\bigl( |X - x_0|\geq t/(8\EE|\nabla f|)\bigr)\\
&\leq  \PP\bigl(|X-\EE X| \geq t/(8\EE|\nabla f|) - |x_0-\EE X|\bigr)\\
& \leq \PP\bigl(|X-\EE X| \geq t/(8\EE|\nabla f|) -M\bigr).
\end{align*}
If now $t/(16\EE |\nabla f|) \geq 2 M$, we can conclude from~\eqref{eq:EX} that
\begin{equation*}
\PP(f\leq -t) \leq \PP\bigl(|X-\EE X| \geq M + t/(16\EE|\nabla f|)\bigr) \leq 8e^{- t\sqrt{\lambda}/(32\EE|\nabla f|)},
\end{equation*}
which ends the proof.
\end{proof}

\subsection{Infimum convolution. Dual formulation of transportation inequalities}
We will rely on the following lemma proved in \cite{gozlan_new} (and in a slightly different version also in \cite{gozlan}). The proof in \cite{gozlan_new} is presented for the real line, but it is not difficult to see that it generalizes to arbitrary dimension.

\begin{lemma}\label{le:inf-convolution}
Let $\theta\colon\RR^n \to \RR_+$ be a convex cost function, $\theta(0) = 0$, $\lim_{x\to \infty} \theta(x) = \infty$. For all functions $f\colon \RR^n \to \RR$ bounded from below, $x \in \RR^n$, and $t > 0$ set
\begin{displaymath}
Q_tf(x)  = Q_t ^{\theta}f(x) = \inf_{y\in \RR^n} \big\{f(y) + t\theta\Bigl(\frac{x-y}{t}\Bigr)\big\}.
\end{displaymath}
Then
\begin{itemize}
\item[(i)] $\mu$ satisfies $\Tbar^+_\theta$ if and only if for all convex $f\colon \RR^n\to \RR$, bounded from below,
\begin{align}\label{eq:dual-T+}
\exp\Big(\int_{\RR^n} Q_1 f d\mu\Big)\int_{\RR^n} e^{-f}d\mu \le 1;
\end{align}
\item[(ii)] $\mu$ satisfies $\Tbar^-_\theta$ if and only if for all convex $f\colon \RR^n\to \RR$, bounded from below,
\begin{align}\label{eq:dual-T-}
\int_{\RR^n}  \exp(Q_1 f )d\mu \ \exp\Big(-\int_{\RR^n} f d\mu\Big) \le 1;
\end{align}
\item[(iii)] if $\mu$ satisfies $\Tbar_\theta$, then for all convex $f\colon \RR^n\to \RR$, bounded from below,
\begin{align}\label{eq:inf-convolution}
\int_{\RR^n}  \exp(Q_t f )d\mu \int_{\RR^n}  e^{-f}d\mu \le 1
\end{align}
holds with $t= 2$. Conversely, if  $\mu$ satisfies \eqref{eq:inf-convolution} for some $t > 0$, then it satisfies $\Tbar_{\tilde{\theta}}$ with $\tilde{\theta}(\cdot) = t\theta(\cdot/t)$.
\end{itemize}
Moreover, the inequality \eqref{eq:dual-T+} (resp. \eqref{eq:dual-T-}) for all convex, Lipschitz functions bounded from below is a sufficient condition for $\Tbar^+_\theta$ (resp. $\Tbar^-_\theta$).
\end{lemma}

The inequality \eqref{eq:inf-convolution} was introduced by Maurey in~\cite{MR1097258} and the relation with transportation cost inequalities was first observed in \cite{MR1682772}.

\section{From convex Poincar\'e to modified log-Sobolev inequalities
for convex and concave functions\label{sec:log-Sobolev}}

In this section we present modified log-Sobolev inequalities for convex and concave functions which are implied by the convex Poincar\'e inequality. Our approach builds heavily on the arguments introduced by Bobkov and Ledoux in \cite{MR1440138} for arbitrary Lipschitz functions, however some non-trivial modifications will be necessary in order to handle the difficulties imposed by the restriction of the Poincar\'e inequality to convex functions.

In what follows for a nonnegative random variable $Y$, we define its entropy as
\begin{displaymath}
\Ent Y = \EE Y\log Y - \EE Y\log(\EE Y)
\end{displaymath}
if $\EE Y\log Y < \infty$ and $\Ent Y = \infty$ otherwise. We refer to e.g \cite{MR1845806,MR1849347} for basic properties of entropy and log-Sobolev inequalities.

Throughout this section we assume that $\mu$ is a probability measure on $\RR^n$ satisfying the convex Poincar\'e inequality \eqref{eq:convex-Poincare} and that $X$ is a random vector with law $\mu$, which will not be explicitly stated in all the theorems.

\subsection{Modified log-Sobolev inequalities for convex functions}

\begin{theorem}\label{thm:BL-convex}
Let $f\colon\RR^n\to\RR$ be convex with  $|\nabla f(x)| \leq c \leq 0.5 \sqrt{\lambda}$ for all $x \in \RR^n$. Then
\begin{equation}\label{eq:Bobkov-Ledoux-inequality-convex}
\Ent( e^{f(X)}) \leq C \EE |\nabla f(X)|^2 e^{f(X)},
\end{equation}
where
\begin{equation*}
C =C(\lambda, c) = \frac{1}{3\lambda} \exp(c\sqrt{2/\lambda}) + \frac{1}{3\bigl(\sqrt{\lambda/2} - c/2\bigr)^2}.
\end{equation*}
\end{theorem}

Our constants are slightly worse than in \cite{MR1440138}, basically because we need to work with the median rather than the mean. However the argument (which works also in the classical case) seems to  slightly simplify the technicalities of \cite{MR1440138}. The proof relies on two propositions.

\begin{proposition}\label{prop:1}
Let $f\colon\RR^n\to\RR$ be convex with $\Med f(X) =0$ and $|\nabla f(x)| \leq c \leq 0.5\sqrt{\lambda}$ for all $x \in \RR^n$. Then
\begin{equation*}
\EE f(X)^2e^{f(X)} \leq C_1 \EE |\nabla f(X)|^2 e^{f(X)},
\end{equation*}
where $C_1 = C_1(c, \lambda) =\bigl( \sqrt{\lambda/2} - c/2\bigr)^{-2}$.
\end{proposition}

\begin{proof}
For $x\in\RR$ we define $\Psi(x) = xe^{x/2}$ and
\begin{equation*}
\Phi(x) = \begin{cases}
xe^{x/2} & \text{for } x \geq -2,\\
-2/e & \text{for } x <-2.
\end{cases}
\end{equation*}
One easily checks that $|\Psi(x)| \leq |\Phi(x)|$, $|\Phi'(x)|\leq |\Psi'(x)|$, and $\Phi$ is convex nondecreasing.

Denote $a^2 = \EE  |\Phi(f(X))|^2 $ and $b^2 = \EE |\nabla f(X)|^2 e^{f(X)}$ (where $a,b\geq 0$). The function $\Phi(f)$ is convex, moreover  $\Med \Phi(f(X)) = 0$. Hence, by Lemma \ref{le:Poincare-median},
\begin{align*}
a^2 &\leq \frac{2}{\lambda} \EE |\nabla f(X)|^2 (1+f(X)/2)^2 e^{f(X)} \indbr{f(X)\geq -2}\\
&\leq  \frac{2}{\lambda}\Bigl(b^2 + c\EE|\nabla f(X)|e^{f(X)/2}\cdot |f(X)| e^{f(X)/2} + \frac{c^2}{4}   \EE f(X)^2  e^{f(X)} \Bigr)\\
&\leq  \frac{2}{\lambda}\Bigl(b^2 + c b \sqrt{\EE f(X)^2 e^{f(X)}} + \frac{c^2}{4} a^2 \Bigr) \\
& \leq \frac{2}{\lambda} \bigl(b + ca/2 \bigr)^2.
\end{align*}
Note that $a<\infty$ (by Proposition~\ref{prop:Poincare-upper-tail} and since $c\leq 0.5\sqrt{\lambda}$). Thus $a(\sqrt{\lambda/2} - c/2)\leq b$ and the assertion follows.
\end{proof}

\begin{proposition}\label{prop:2}
Let $f\colon \RR^n\to\RR$ be either convex or concave, with $\Med f(X) =0$ and $|\nabla f(x)| \leq c $ for all $x\in\RR^n$. Then
\begin{equation*}
\EE |\nabla f(X)|^2 \leq C_2 \EE |\nabla f(X)|^2 e^{f(X)},
\end{equation*}
where $C_2 = C_2(c, \lambda) = \exp({c\sqrt{2/\lambda}})$. Consequently,
\begin{equation*}
\EE f(X)^2 \leq  \frac{2}{\lambda} C_2 \EE |\nabla f(X)|^2 e^{f(X)}.
\end{equation*}
\end{proposition}

\begin{proof}  If $|\nabla f(X)|$ vanishes with probability one, there is nothing to prove. Otherwise, denote by $\widetilde{\EE}$ the expectation with respect to the probability measure with density $|\nabla f(X)|^2/\EE|\nabla f(X)|^2$ relative to $\PP$. By Jensen's inequality,
\begin{equation*}
\EE |\nabla f(X)|^2 e^{-|f(X)|} = \EE|\nabla f(X)|^2 \widetilde{\EE} e^{-|f(X)|} \geq \EE|\nabla f(X)|^2  e^{-\widetilde{\EE}|f(X)|}.
\end{equation*}
Thus, using the trivial inequality $-|f|\leq f$, we conclude that
\begin{equation*}
\EE |\nabla f(X)|^2 \leq e^{\widetilde{\EE}|f(X)|}  \EE |\nabla f(X)|^2 e^{f(X)}.
\end{equation*}
But since
\begin{align*}
\EE |\nabla f(X)|^2 |f(X)| &\leq c \EE |\nabla f(X)| |f(X)| \leq c \sqrt{\EE |\nabla f(X)|^2} \sqrt{\EE f(X)^2 } \\
&\leq c\sqrt{2/\lambda}\EE |\nabla f(X)|^2,
\end{align*}
we can bound $\widetilde{\EE}|f(X)|$ by $c\sqrt{2/\lambda}$. This yields the assertion of the proposition.
\end{proof}

\begin{proof}[Proof of Theorem~\ref{thm:BL-convex}]
Without loss of generality assume $\Med f(X) = 0$. Denote $F(t) = \EE f(X)^2 e^{tf(X)}$, $t\in[0,1]$.
By the formula $\int_0^1 ta^2 e^{ta} dt = ae^a -e^a +1$ and  the convexity of $t\mapsto F(t)$,
\begin{align*}
\Ent (e^{f(X)})
&\leq \EE (f(X)e^{f(X)} - e^{f(X)} + 1) = \EE \int_0^1  tf(X)^2e^{tf(X)} dt = \int_0^1 t F(t) dt \\
& \leq \int_0^1 t(1-t) F(0) + t^2 F(1) dt = \frac{1}{6} F(0) + \frac{1}{3} F(1)
\end{align*}
(note that for this argument to work we do \emph{not} need the expectation of $f(X)$ to vanish).
Thus Propositions~\ref{prop:1} and~\ref{prop:2} imply the assertion of the theorem.
\end{proof}

\subsection{Modified log-Sobolev inequalities for concave functions}

\begin{theorem}\label{thm:BL-concave-n}
Let $f\colon \RR^n \to \RR$ be convex with  $|\nabla f(x)| \leq c < \sqrt{\lambda}/64$ for all $x \in \RR^n$. Assume that $M\in \RR_+$ satisfies $\PP(|X-\EE X|\le M)\ge 3/4$. Then
\begin{equation*}
\Ent (e^{-f(X)}) \leq C \EE |\nabla f(X)|^2 e^{-f(X)},
\end{equation*}
where $C = C(\lambda, c, M)$ is a constant depending only on $\lambda,c,M$.
\end{theorem}

\begin{remark}
If we denote by $X_1,\ldots,X_n$ the coordinates of $X$, then by the Poincar\'e inequality we have
\begin{displaymath}
\EE |X - \EE X|^2 = \sum_{i=1}^n \EE |X_i - \EE X_i|^2 \le \frac{n}{\lambda},
\end{displaymath}
and hence, by the Chebyshev inequality, $M = 2\sqrt{n/\lambda}$ satisfies $\PP(|X-\EE X|\le M)\ge 3/4$. Thus in fixed dimension $n$ and for say $c = \sqrt{\lambda}/128$, the constant $C$ in Theorem~\ref{thm:BL-concave-n} can be bounded uniformly over all probability measures satisfying the convex Poincar\'e inequality with constant $\lambda$.
\end{remark}

\begin{proof}[Proof of Theorem~\ref{thm:BL-concave-n}]
We start as in the proof of Theorem~\ref{thm:BL-convex}. Denote $g=-f$ (this is a \emph{concave} function). Without loss of generality assume $\Med g(X) = 0$. Denote $F(t) = \EE g(X)^2 e^{tg(X)}$, $t\in[0,1]$.
By the convexity of $t\mapsto F(t)$,
\begin{align}
\Ent (e^{g(X)})
&\leq \EE (g(X)e^{g(X)} - e^{g(X)} + 1) = \EE \int_0^1  tg(X)^2e^{tg(X)} dt = \int_0^1 t F(t) dt \nonumber\\
& \leq \int_0^1 t(1-t) F(0) + t^2 F(1) dt = \frac{1}{6} F(0) + \frac{1}{3} F(1).\label{eq:FF1-concave}
\end{align}
We have
\begin{equation}\label{eq:F1-concave}
F(1) \leq \EE g(X)^2 + \EE g_+(X)^2 e^{g_+(X)} = F(0) + \EE g_+(X)^2 e^{g_+(X)}
\end{equation}
By Proposition~\ref{prop:2}, $F(0) \leq \frac{2}{\lambda} \exp(c\sqrt{2/\lambda}) \EE |\nabla g(X)|^2 e^{g(X)}$, so it remains to estimate $\EE g_+(X)^2 e^{g_+(X)}$.

Integration by parts and Lemma~\ref{le:conc-Med-concave-n} yield
\begin{align*}
\EE e^{2g_+(X)}
&= 1+ \int_0^{\infty} 2e^{2t} \PP(g_+(X) \geq t) dt \\
&= 1+  \int_0^{32Mc} 2e^{2t} dt + \int_{32Mc}^{\infty} 2e^{2t} \PP(g_+(X) \geq t) dt \\
& \leq e^{64Mc} + \int_{32Mc}^{\infty}  16e^{2t - t\sqrt{\lambda}/(32c)} dt < D_1 = D_1(\lambda, c,M) <\infty,
\end{align*}
if only $c < \sqrt{\lambda}/64$. Similarly (using  Lemma~\ref{le:conc-Med-concave-n} in its full strength),
\begin{align*}
\EE g_+(X)^4
&= \int_0^{\infty} 4t^3 \PP(g_+(X) \geq t) dt \\
&=\int_0^{32M\EE|\nabla f(X)|} 4t^3 dt + \int_{32M\EE|\nabla f(X)|}^{\infty} 4t^3\PP(g_+(X) \geq t) dt \\
& \leq (32M\EE|\nabla f(X)|)^4 + 4 \int_{32M\EE|\nabla f(X)|}^{\infty}  t^3e^{- t\sqrt{\lambda}/(32\EE|\nabla f(X)|)} dt\\
&\leq D_2 (\EE|\nabla f(X)|)^4 \leq  D_2 (\EE |\nabla f(X)|^2) ^2
\end{align*}
 for some $D_2 = D_2(\lambda, M)$.  Thus,  by Proposition~\ref{prop:2},
 \begin{align*}
 \EE g_+(X)^2 e^{g_+(X)} &\leq \sqrt{\EE g_+(X)^4} \sqrt{\EE e^{2g_+(X)}} \leq \sqrt{D_1D_2} \EE|\nabla f(X)|^2 \\
 &\leq \sqrt{D_1D_2} e^{c\sqrt{2/\lambda}}\EE|\nabla f(X)|^2 e^{f(X)}.
 \end{align*}
This, together with~\eqref{eq:FF1-concave} and~\eqref{eq:F1-concave}, ends the proof:
\begin{align*}
\Ent (e^{-f(X)})
& \leq \frac{1}{6} F(0) + \frac{1}{3} F(1) \leq \frac{1}{2} F(0) + \frac{1}{3}  \EE g_+(X)^2 e^{g_+(X)}\\
& \leq \bigl(\frac{1}{\lambda} + \frac{1}{3}\sqrt{D_1 D_2} \bigr) e^{c\sqrt{2/\lambda}}\EE |\nabla f(X)|^2 e^{-f(X)}.\qedhere
\end{align*}
\end{proof}

\section{Proof of the main result}\label{sec:main-proof}

We will now present the proof of Theorem~\ref{thm:main}. As already mentioned, the implication (ii) $\implies$ (i) is standard, we provide a sketch of its proof just for the sake of completeness. The proof of the implication (i) $\implies$ (ii) follows the arguments introduced first in~\cite{MR1846020} and based on the analysis of the Hamilton-Jacobi equations. A~crucial element of the proof will be the modified log-Sobolev inequalities obtained in Section~\ref{sec:log-Sobolev}.

\begin{lemma}\label{lem:logSob-transport}
Let $X$ be a random vector in $\RR^n$. Assume that there exist $C<\infty$ and $L>0$ such that
\begin{align}\label{eq:exponential-integrability}
\EE e^{L|X|} < \infty
\end{align}
and the inequality
\begin{equation}
\label{eq:logSob-transport}
\Ent(e^{f(X)}) \leq C \EE |\nabla f(X)|^2 e^{f(X)}
\end{equation}
holds for every convex (respectively: concave) $L$-Lipschitz function $f\colon\RR^n\to \RR$. Then, for every convex Lipschitz function $f\colon\RR^n\to\RR$ bounded from below,
\begin{align*}
\EE e^{Q_1^{\alpha}f(X)} e^{-\EE f(X)} &\leq 1\\
\big( \text{respectively: } \quad
e^{ \EE Q_1^{\alpha} f(X) } \EE e^{- f(X)} &\leq 1\big),
\end{align*}
where $Q_t^{\alpha} f(x) = \inf_{y\in\RR^n} \{ f(x-y) + t\alpha(|y|/t) \}$, $t>0$,
 is the infimum convolution operator with the cost function
\begin{equation}\label{eq:cost-alpha}
\alpha(s) =\begin{cases}
\frac{s^2}{4C} & \text{for } |s| \le 2CL,\\
L|s| - L^2C  & \text{for } |s| > 2CL.
\end{cases}
\end{equation}
\end{lemma}

\begin{remark}
The condition \eqref{eq:exponential-integrability} is introduced to exclude heavy-tailed measures for which the only exponentially integrable convex functions are constants. Note that in this case the inequality \eqref{eq:logSob-transport} is trivially satisfied, while the transportation inequality cannot hold (as it implies the existence of exponential moments).
\end{remark}

If we  recall the dual formulations of the weak transport-entropy inequalities $\Tbar^-$ and $\Tbar^+$ (see Lemma~\ref{le:inf-convolution}),  the definition of $\theta_{C,D}$ from~\eqref{eq:cost-function}, and the results of the preceding section (namely, Theorems~\ref{thm:BL-convex} and~\ref{thm:BL-concave-n}), we immediately obtain the following corollaries.

\begin{corollary}\label{cor:Poincare-transport-minus}
Let $X$ be a random vector in $\RR^n$ satisfying the convex Poincar\'e inequality \eqref{eq:convex-Poincare}. Then, for any $c\leq 0.5\sqrt{\lambda}$, the law of $X$ satisfies the inequality $\Tbar^-_{\theta_{2C,c}}$ with \begin{equation*}
C =C(\lambda, c) = \frac{1}{3\lambda} \exp(c\sqrt{2/\lambda}) + \frac{1}{3\bigl(\sqrt{\lambda/2} - c/2\bigr)^2}.
\end{equation*}
\end{corollary}

\begin{corollary}\label{cor:Poincare-transport-plus}
Let $X$ be a random vector in $\RR^n$ satisfying the convex Poincar\'e inequality \eqref{eq:convex-Poincare} and let $M$ be any number such that $\PP(|X-\EE X|\leq M)\geq 3/4$. Then, for any $c < \sqrt{\lambda}/64$, the law of $X$ satisfies the inequality $\Tbar^+_{\theta_{2C,c}}$  for some constant $C =C(\lambda, c, M) $ depending only on $\lambda$, $c$, and $M$.
\end{corollary}

\begin{proof}[Proof of Lemma~\ref{lem:logSob-transport}]
Suppose that the log-Sobolev inequality~\eqref{eq:logSob-transport} holds for all \emph{convex} and $L$-Lipschitz functions. We first present a perturbation argument which allows us to work with random vectors with an absolutely continuous law. We then shall follow the approach of \cite[Proof of Theorem 1.5]{gozlan_new}.

Let $G$ be a Gaussian random vector in $\RR^n$, independent of $X$, with the covariance matrix being a~sufficiently small multiple of identity, so that it satisfies the usual log-Sobolev inequality with constant $C$,
\begin{displaymath}
\Ent e^{f(G)} \le C\EE |\nabla f(G)|^2 e^{f(G)}
\end{displaymath}
for all Lipschitz functions $f \colon \RR^n \to \RR$ (see e.g. Theorem 5.1. in \cite{MR1849347} for an equivalent formulation).

 Then, by the tensorization property of entropy (see e.g.  Proposition 5.6. in \cite{MR1849347}), the random vector $(X,G)$ on $\RR^n \times \RR^n$ satisfies the modified log-Sobolev inequality
\begin{equation}
\label{eq:logSob-tr-2}
\Ent(e^{F(X,G)}) \leq C \EE (|\nabla_X F(X,G)|^2 + |\nabla_G F(X,G)|^2)e^{F(X,G)}
\end{equation}
for all convex functions $F\colon\RR^n\times\RR^n \to\RR$ which are $L$-Lipschitz with respect to the first coordinate (here $|\nabla_X F|$ and $|\nabla_G F|$ denote partial lengths of gradients with respect to the first and second variable, with the other variable fixed).

 Let $f\colon\RR^n \to \RR$ be a convex $L$-Lipschitz function and consider $\varepsilon > 0$. Applying the inequality~\eqref{eq:logSob-tr-2} to the function defined by the formula $F (x,y) = f(x+\eps y)$ for $x,y\in\RR^n$ (which is $L$-Lipschitz with respect to the first variable),
  we see  that the random vector $X_{\eps} = X+ \varepsilon G$ satisfies the modified log-Sobolev inequality
\begin{equation}
\label{eq:logSob-tr-3}
\Ent(e^{f(X_{\eps})}) \leq C_{\eps} \EE |\nabla f(X_{\eps})|^2 e^{f(X_{\eps})},
\end{equation}
where $C_\varepsilon = C(1+\eps^2)$. Note that the law of $X_\varepsilon$ is absolutely continuous with respect to the Lebesgue measure on $\RR^n$, and so almost surely $X_\eps$ is a differentiability point of $f$ and $|\nabla f(X_\eps)|$ coincides with the Euclidean length of the `true' gradient $\nabla f(X_\eps)$.

Moreover, \eqref{eq:logSob-tr-3} can be rewritten in the form
\begin{equation}
\label{eq:logSob-transport-4}
\Ent(e^{f(X_{\eps})}) \leq \EE \alpha_{\eps}^*( |\nabla f(X_{\eps})|) e^{f(X_{\eps})},
\end{equation}
where
\begin{equation*}
\alpha_{\eps}^*(s) =\begin{cases}
C_{\eps}|s|^2 & \text{for } |s| \le L,\\
+\infty  & \text{for } |s|>L.
\end{cases}
\end{equation*}
is the Legendre transform of $\alpha_{\eps}(s) = \min\{\frac{s^2}{4C_{\eps}},
L|s| - L^2C_{\eps} \}$.

If  $f\colon\RR^n\to\RR$ is convex, Lipschitz (with arbitrary Lipschitz constant)
 and bounded from below, then $Q_t^{\alpha_{\eps}} f$
is well defined,
convex (as an infimum convolution of two convex functions),
 and $L$-Lipschitz   for $t\in(0,1]$ (since  $Q_t^{\alpha_{\eps}}f(x) = \inf_{y\in\RR^n} \{ f(y) + t\alpha_{\eps}(|y-x|/t) \}$ and the function $x\mapsto  t\alpha_{\eps}(|y-x|/t) $ is $L$-Lipschitz  for $t\in(0,1]$).

Moreover, the function $u(t,x) = Q_t^{\alpha_{\eps}} f(x) $ is Lipschitz on $(0,\infty)\times\RR^n$ and satisfies the Hamilton-Jacobi equation
\begin{equation*}
\frac{d}{d t} u(t,x) + \alpha_{\eps}^*(|\nabla_x u(t,x)|) =0 \quad \text{for Lebesgue almost all} \;(t,x) \in (0,\infty)\times \RR^n,
\end{equation*}
(see Proposition~\ref{prop:HJ} in Appendix~\ref{sec:HJ}). Set
\begin{equation*}
F(t) = \frac{1}{t} \ln\bigl( \EE e^{tQ_t^{\alpha_{\eps}} f(X_{\eps})} \bigr), \quad t\in(0,1].
\end{equation*}
(Note that $F(t)<\infty$ since $Q_t^{\alpha_{\eps}} f$ is $L$-Lipschitz.) Using the integrability properties of $X$ (and as a consequence of $X_\varepsilon$), together with the Lipschitz property of $u$ it is not difficult to see that $F$ is locally Lipschitz  and  for Lebesgue almost all
$t \in (0,1)$,
\begin{align*}
\frac{d}{dt}F(t) &=
-\frac{1}{t^2} \ln\bigl( \EE e^{tQ_t^{\alpha_{\eps}} f(X_{\eps})} \bigr)+ \frac{1}{t}\frac{ \EE e^{tQ_t^{\alpha_{\eps}} f(X_{\eps})} \bigl(Q_t^{\alpha_{\eps}} f(X_{\eps})
+ t \frac{d}{dt} Q_t^{\alpha_{\eps}} f(X_{\eps}) \bigr)}{ \EE e^{tQ_t^{\alpha_{\eps}} f(X_{\eps})} }\\
& = \frac{1}{t^2 \EE e^{tQ_t^{\alpha_{\eps}} f(X_{\eps}) } }
\Bigl(\Ent\bigl( e^{tQ_t^{\alpha_{\eps}} f(X_{\eps}) }\bigr)
 - t^2\EE\alpha_{\eps}^*(|\nabla Q_t^{\alpha_{\eps}} f(X_{\eps})|) e^{tQ_t^{\alpha_{\eps}} f(X_{\eps}) }\Bigr)\\
& \leq  \frac{1}{ \EE e^{tQ_t^{\alpha_{\eps}} f(X_{\eps}) }}
C_{\eps}\EE \bigl(|\nabla Q_t^{\alpha_{\eps}} f(X_{\eps}) |^2  -|\nabla Q_t^{\alpha_{\eps}} f(X_{\eps})|^2\bigr)e^{tQ_t^{\alpha_{\eps}} f(X_{\eps}) } = 0,
\end{align*}
where we used \eqref{eq:logSob-transport-4}, the definition of $\alpha_{\eps}^*$, and the fact that $Q_t^{\alpha_{\eps}} f$ is $L$-Lipschitz. Thus
\begin{equation*}
F(1) \leq \liminf_{t\to 0^+} F(t) \leq  \lim_{t\to 0^+}  \frac{\ln\bigl( \EE e^{t f(X_{\eps})} \bigr)}{t}  = \EE f(X_{\eps}),
\end{equation*}
or, in other words,
\begin{equation*}
\EE e^{ Q_1^{\alpha_{\eps}}f(X_{\eps})} \leq e^{ \EE f(X_{\eps})}.
\end{equation*}
It is easy to see that by taking $\varepsilon\to 0$ we arrive at the assertion of the lemma (recall that  $f$ and $Q_1^{\alpha_{\eps}}$ are Lipschitz and $\alpha_{\eps}\leq \alpha$).

Suppose now that the log-Sobolev inequality~\eqref{eq:logSob-transport} holds for all \emph{concave} and $L$-Lipschitz functions.
As before, we pass to the random vector $X_{\eps}$ which has  an absolutely continuous  distribution. Let $g\colon\RR^n\to\RR$ be convex and bounded from below. Then the function $f= -Q_1^{\alpha_\eps} g$ is concave and $L$-Lipschitz. The same calculation as above yields
\begin{equation*}
\EE e^{ Q_1^{\alpha_{\eps}}f(X_{\eps})} \leq e^{ \EE f(X_{\eps})},
\end{equation*}
or equivalently
\begin{equation*}
\EE e^{ Q_1^{\alpha_{\eps}} (- Q_1^{\alpha_{\eps}}g)(X_{\eps})} \leq e^{ -\EE  Q_1^{\alpha_{\eps}}g(X_{\eps})}.
\end{equation*}
We stress that now, in order to prove the Hamilton-Jacobi equations via Proposition \ref{prop:HJ}, we need to use the $L$-Lipschitz property of $f$, since in general $f$ is not bounded from below.

Since
\begin{equation*}
- g(x)\leq \inf_{y\in\RR^n} \sup_{z\in\RR^n} \{ -g(z) - \alpha_{\eps}(|z-y|) + \alpha_{\eps}(|y-x|) \} =Q_1^{\alpha_{\eps}} (- Q_1^{\alpha_{\eps}})g(x)
\end{equation*}
(to verify the inequality take $z=x$), a limit argument yields the assertion.
\end{proof}

We are now ready for the proof of our main result.

\begin{proof}[Proof of Theorem~\ref{thm:main}] The implication (i)$\implies$(ii) follows immediately from   Corollaries~\ref{cor:Poincare-transport-minus} and~\ref{cor:Poincare-transport-plus}, and the definition of $\Tbar_{\theta_{2C,c}}$. To obtain the reverse implication one can use a standard Taylor expansion argument. Assume that $\Tbar_{\theta_{C,D}}$ holds. Let $f\colon\RR^n \to \RR$ be convex, Lipschitz, and bounded from below. For $x\in\RR^n$ denote
\begin{equation*}
f^x (z) = f(x) + \langle u_x,  z-x\rangle, \quad z\in\RR^n,
\end{equation*}
where $u_x$ is any subgradient of $f$ at $x$, so that $f^x \le f$ on $\RR^n$. Taking $z = x + \varepsilon u_x$ with $\varepsilon \to 0$ we see that $|u_x| \le |\nabla f(x)|$.

For sufficiently small $\eps$ we have $\eps |\nabla f(x)|\leq D$ for all $x\in\RR^n$, and hence
\begin{align*}
Q_2^{\theta_{C,D}} (\eps f)(x)
&\geq \inf_{y\in\RR} \{ \eps f^x(x-y) + 2\theta_{C,D}(y/2) \}\\
&= \eps f(x) +\inf_{y\in\RR} \{ -\eps \langle u_x,y \rangle + 2\theta_{C,D}(y/2)\}\\
& = \eps f(x) - 2\theta_{C,D}^*(\eps u_x) \geq  \eps f(x) - \eps^2 C|\nabla f(x)|^2
\end{align*}
(recall that $|u_x|\leq|\nabla f(x)|$).
We now substitute $\eps f$ into the dual formulation~\eqref{eq:inf-convolution} and use the above estimate. An inspection of the Taylor expansions up to order $\eps^2$ yields
\begin{equation*}
\Var(f(X)) \leq C \EE |\nabla f(X)|^2.
\end{equation*}
 This ends the proof.
 \end{proof}

\section{Examples of measures satisfying the convex Poincar\'e inequality}\label{sec:Examples}

We will now discuss several tools which allow to construct measures satisfying the convex Poincar\'e inequality. To shorten the notation we will denote by $\EE_\mu$ and $\Var_\mu$ respectively the mean and variance of $f$ seen as a random variable on $\RR^n$ equipped with probability measure $\mu$.

Let us start with the well known tensorization  property of variance (see e.g. \cite[Proposition 1.4.1]{MR1845806}), which asserts that whenever $\mu_i$ are probability measures on $\mathcal{X}_i$, $i=1,\ldots,n$, then the product measure $\mu = \mu_1\otimes\cdots\otimes \mu_n$ on $\mathcal{X}_1\times\cdots\times \mathcal{X}_n$, satisfies the inequality
\begin{displaymath}
\Var_{\mu} f \le \sum_{i=1}^n \EE_\mu \Var_{\mu_i} f,
\end{displaymath}
for every function $f \colon \mathcal{X}_1\times\cdots\times \mathcal{X}_n \to \RR$, where $\Var_{\mu_i} f$ denotes the variance of $f$ treated as a function on $\mathcal{X}_i$, with the other coordinates fixed.

This immediately implies the tensorization property for the convex Poincar\'e inequality, namely if $\mu_i$ ($i=1,\ldots,N$) is a probability measure on $\RR^{n_i}$, satisfying the convex Poincar\'e  inequality with constant $\lambda$, then the product measure $\mu = \mu_1\otimes\cdots\otimes \mu_N$ on $\RR^{n_1+\cdots+n_N}$ satisfies
\begin{align}\label{eq:tensorization}
\Var_\mu f \le \frac{1}{\lambda} \EE \sum_{i=1}^n |\nabla_i f|^2,
\end{align}
for every convex function $f\colon \RR^{n_1+\cdots+\mu_n} \to \RR$, where $|\nabla_i f|$ denotes the `partial length of gradient' along $\RR^{n_i}$.
If the measures $\mu_i$ are absolutely continuous with respect to the Lebesgue measure, then by Rademacher's theorem locally Lipschitz functions are almost everywhere differentiable, in particular the right-hand side of the above inequality coincides with $\lambda^{-1}\EE |\nabla f|^2$ and so we obtain that $\mu$ satisfies the convex Poincar\'e inequality with constant $\lambda$.
The situation is more delicate for measures which are not absolutely continuous, however thanks to results by Gozlan, Roberto and Samson \cite{MR3311918}, we can obtain the following simple proposition.

\begin{proposition}\label{prop:tensorization}
Assume that $\mu_i$ are probability measures on $\RR^{n_i}$, $i=1,\ldots,n$, satisfying the convex Poincar\'e inequality with constant $\lambda$. Then the measure $\mu = \mu_1\otimes \cdots\otimes \mu_n$ on $\RR^{n_1+\cdots+n_N}$ satisfies the convex Poincar\'e inequality with constant $\lambda/C$ for some universal constant $C$
\end{proposition}

\begin{proof}We provide only a sketch of the proof, leaving some computational details to the Reader.
Denote $n = n_1+\cdots+n_N$ and consider an arbitrary convex smooth 1-Lipschitz function $f$ on $\RR^{nk}$, $k \ge 1$. By \eqref{eq:tensorization} we have $\Var_{\mu^{\otimes k}} f \le \lambda^{-1}\EE_{\mu^{\otimes k}}|\nabla f|^2 \le 1$. Using an analogous argument as in the proof of Proposition \ref{prop:Poincare-upper-tail} (for $p > 2$, to remain in the smooth setting) we arrive at
\begin{align}\label{eq:dimension-free}
\mu^{\otimes k}(f \ge \Med_\mu f + t) \le 8e^{-\sqrt{\lambda}t/2}
\end{align}
for all 1-Lipschitz smooth convex functions. We can extend this inequality to arbitrary 1-Lipschitz convex function (approximating them with 1-Lipschitz smooth convex functions, e.g. by convolving them with Gaussian densities, see \cite[p.~429]{MR1756011}), so in particular we get that for any convex set $A \subseteq \RR^{nk}$, with $\mu^{\otimes k}(A) \ge 1/2$, and all $t > 0$,
\begin{displaymath}
\mu^{\otimes k}(A + tB_2^{nk}) \ge 1 - 8e^{-\sqrt{\lambda}t/2},
\end{displaymath}
where $B_2^{nk}$ is the unit Euclidean ball in $\RR^{nk}$. Recall the notation
\begin{displaymath}
|\nabla^- f(x)| = \limsup_{y\to x} \frac{(f(y) - f(x))_-}{|x-y|}
\end{displaymath}

By \cite[Theorem 6.7]{MR3311918}, the dimension free subexponential concentration for convex sets of the form \eqref{eq:dimension-free} implies that $\mu$ satisfies the Poincar\'e inequality
\begin{align}\label{eq:Poincare-minus}
\Var_\mu f \le \frac{1}{\lambda'} \EE \sum_{i=1}^n |\nabla_i^- f|^2 \le \frac{1}{\lambda'} \EE \sum_{i=1}^n |\nabla_i f|^2
\end{align}
for all convex functions $f$, where
\begin{displaymath}
\sqrt{\lambda'} = \sup\Big\{\frac{\bar{\Phi}^{-1}(8\exp(-\sqrt{\lambda}r/2))}{r}\colon r \ge \frac{2\log(16)}{\sqrt{\lambda}} \Big\},
\end{displaymath}
where $\bar{\Phi}$ is the Gaussian tail function. Using the estimate $\bar{\Phi}(x) \ge \frac{1}{2} e^{-x^2}$ and performing some elementary calculations, we arrive at
the assertion of the proposition.
\end{proof}

\begin{remark} The above argument shows that if $\mu$ satisfies the Poincar\'e inequality \eqref{eq:convex-Poincare} then it also satisfies the formally stronger inequality \eqref{eq:Poincare-minus} with $\lambda'=\lambda/C$. We remark that in the category of all Lipschitz functions it is known that the Poincar\'e inequalities with the length of gradients $|\nabla^- f|$ and $|\nabla f|$ are equivalent and the involved constants do not change (cf. \cite[Remark 1.1]{MR3311918}).
\end{remark}

Tensorization allows in particular to pass from one-dimensional measures satisfying the convex Poincar\'e inequality (characterized in \cite{bobkov-goetze}) to product measures in higher dimensions. Another standard tool for producing new examples is perturbation: if $\mu$ satisfies the convex Poincar\'e inequality with constant $\lambda$ and $\nu$ is a measure with density $e^U$ with respect to $\mu$, then $\nu$ satisfies the convex Poincar\'e inequality with constant $\lambda \exp(\inf U - \sup U)$. For the proof see e.g. \cite[Chapter 3.4]{MR1845806} (the proof therein is written in the context of Markov processes and Dirichlet forms but it is based only on the elementary observation that $\Var f = \inf_{a \in \RR} \EE |f - a|^2$ and works in exactly the same way in the convex setting).

Perturbation and tensorization are tools that appeared for the first time in the `classical' theory of Poincar\'e and log-Sobolev inequalities for smooth (or locally Lipschitz) functions. The next proposition does not have a counterpart in the classical setting and significantly extends the set of tools for creating new examples. Namely, we will show that the convex Poincar\'e inequality passes to mixtures of measures. Note that this cannot be the case for the classical Poincar\'e inequality since it clearly cannot hold for measures with disconnected support. We note however that the preservation of the Poincar\'e and log-Sobolev inequalities by mixtures of measures with overlapping supports has been investigated by Chafa\"{\i} and Malrieu in \cite{MR2641771}. In particular, the Proposition \ref{prop:mixtures} below has been inspired by calculation in Section 4.1 therein.

Let $\mathcal{T}_2(\mu_0,\mu_1)$ stand for the usual Kantorovich transport cost between $\mu_1$ and $\mu_0$ (defined by taking $\theta(x) = |x|^2$ in~\eqref{eq:standard-cost}), in other words the square of the Kantorovich-Wasserstein distance $W_2$.

\begin{proposition}\label{prop:mixtures}
Let $\mu_0$, $\mu_1$ be probability measures on $\RR^n$ which satisfy the convex Poincar\'e inequality~\eqref{eq:convex-Poincare} with constants $\lambda_0$ and $\lambda_1$ respectively. Then the measure $\mu_p = p\mu_1 + (1-p)\mu_0$, $p\in(0,1)$, satisfies the convex Poincar\'e inequality~\eqref{eq:convex-Poincare} with constant
\begin{equation*}
\lambda' = \bigl(\max\{1/\lambda_1, 1/\lambda_0 \} + 2\mathcal{T}_2(\mu_0,\mu_1)\bigr)^{-1}.
\end{equation*}
\end{proposition}

\begin{proof}

If $f\colon\RR^n\to\RR$ is a convex function, then
\begin{align*}
\Var_{\mu_p} (f)
&= p\Var_{\mu_1} (f) + (1-p)\Var_{\mu_0} (f) + p(1-p)(\EE_{\mu_1} f  - \EE_{\mu_0} f)^2 \\
&\leq \max\{1/\lambda_1, 1/\lambda_0\} \EE_{\mu_p} |\nabla f|^2 + p(1-p)(\EE_{\mu_1} f  - \EE_{\mu_0} f)^2
\end{align*}
and it suffices to estimate the last term.

Let $X$ and $Y$ be random vectors in $\RR^n$ with laws $\mu_1$ and $\mu_0$ respectively. By convexity of $f$,
\begin{align*}
\bigl| \EE f(X) -\EE f(Y) \bigr|
&\leq \EE (|\nabla f(X)|+ |\nabla f(Y)|) |X-Y|\\
&\leq (\sqrt{\EE |\nabla f(X)|^2}  + \sqrt{\EE |\nabla f(Y)|^2}) \sqrt{\EE |X-Y|^2}.
\end{align*}
Thus,
\begin{align*}
p(1-p)(\EE_{\mu_1} f  - \EE_{\mu_0} f)^2
&\leq 2p(1-p) \bigl(\EE_{\mu_1} |\nabla f|^2 +  \EE_{\mu_0} |\nabla f|^2\bigr) \EE |X-Y|^2\\
&\leq 2 \EE |X-Y|^2 \EE_{\mu_p} |\nabla f|^2.
\end{align*}
Taking the infimum over all realizations of $X$ and $Y$ implies the assertion.
\end{proof}

\section{Refined concentration of measure derived from infimum convolution inequalities}\label{sec:concentration}

In this section we explain what concentration inequalities for convex functions can be obtained from general infimum convolution inequalities of the form \eqref{eq:inf-convolution}. While some parts of our derivation are well known and are included only for the sake of completeness, we also provide new inequalities valid beyond the setting of Lipschitz functions. Their proofs are elementary but to our best knowledge they have not been noted in the literature before.

Throughout this section $\theta \colon \RR^n \to [0,\infty)$ is a convex function. We also assume the following conditions:
\begin{itemize}
\item $\theta(x) = \theta(-x)$ for all $x \in \RR^n$,
\item  $\theta(x) = 0$ if and only if $x = 0$ (in particular, by convexity, $\lim_{x\to \infty} \theta(x) = \infty$).
\end{itemize}

We remark that at the cost of some technical work one can obtain the results we present below for more general cost functions (e.g. taking the value $\infty$ or not satisfying the symmetry condition). We restrict to the smaller class to simplify the presentation.

In what follows, for a function $f\colon \RR^n \to \RR$, bounded from below, we set
\begin{displaymath}
  Qf(x) = Q_1^\theta f(x) = \inf_{y\in \RR^n} \bigl\{f(y) + \theta(x-y)\bigr\}.
\end{displaymath}
We also denote
\begin{equation*}
B_{\theta}(r) = \{x\in\RR^n : \theta(x) <r \}, \quad r>0.
\end{equation*}

\subsection{Enlargements of sets and concentration for Lipschitz functions}
Let us start with the classical description of concentration of measure in terms of enlargements of sets. The following proposition goes back to \cite{MR1097258}.

\begin{proposition}\label{prop:Maurey}
Assume that $\mu$ is a probability measure on $\RR^n$, satisfying
\begin{align}\label{eq:inf_conv_again}
\int_{\RR^n}  e^{Q f}d\mu \int_{\RR^n}  e^{-f}d\mu \le 1
\end{align}
for all convex functions $f\colon \RR^n \to \RR$, bounded from below. Then for all convex subsets $A\subseteq \RR^n$ and $r > 0$, we have
\begin{displaymath}
\mu\bigl((A+B_{\theta}(r))^c\bigr)\mu(A) \le e^{-r}.
\end{displaymath}
\end{proposition}

\begin{proof}
Consider $f = \infty \ind{(\cl A)^c}$ and note that $Qf(x) < r$ if and only if there exists $y \in A$ such that $\theta(x-y) < r$.
Applying the inequality~\eqref{eq:inf_conv_again} to~$f$ (which can be justified by monotone approximation), we obtain
\begin{displaymath}
e^r \mu\bigl((A+B_{\theta}(r))^c\bigr)\mu(A) \le \int_{\RR^n}  e^{Q f}d\mu \int_{\RR^n}  e^{-f}d\mu \le 1.
\qedhere
\end{displaymath}
\end{proof}

To formulate corollaries to the above proposition we need to introduce new notation, which at first may seem rather abstract. However, as the examples presented in the subsequent parts of this section will show, it will prove useful in providing a uniform framework for concentration inequalities, especially in the non-Lipschitz case.

\begin{definition} Define the norm $|\cdot|_{\frac{1}{p}\theta}$ on $\RR^n$, as the Orlicz norm corresponding to the function $x \mapsto \frac{1}{p}\theta(x)$, i.e.
\begin{displaymath}
  |x|_{\frac{1}{p}\theta} = \inf\{a > 0 \colon \theta(x/a) \le p\}.
\end{displaymath}
Define also the norm $|\cdot|\dualnorm$ on $\RR^n$ as the dual to $|\cdot|_{\frac{1}{p}\theta}$, i.e.
\begin{displaymath}
|x|\dualnorm = \sup\Big\{\sum_{i=1}^n x_iy_i\colon \theta(y) \le p\Big\}.
\end{displaymath}
\end{definition}

The norm $|x|\dualnorm$ is equivalent (up to universal constants) to the Orlicz norm $|\cdot|_{\theta^\ast_p}$ related to the function $\theta^\ast_p(x) = \frac{1}{p}\theta^\ast(px)$, explicitly given by
\begin{displaymath}
  |\cdot|_{\theta^\ast_p} = \inf \{ a> 0\colon \theta^\ast_p(x/a) \le 1\} = \inf \{ a> 0\colon \theta^\ast(px/a) \le p\}.
\end{displaymath}
It was observed by Gluskin and Kwapie{\'n} in \cite{MR1338834} that norms of this type play an important role in moment estimates for sums of independent random variables. Recently it has been noted \cite{Nonlipschitz,ESAIM} that they also appear in moment estimates for smooth functions of random vectors satisfying modified log-Sobolev inequalities. Since in the context of transportation or infimum convolution inequalities one starts from the function $\theta$ and not from $\theta^\ast$ (which is the case in the corresponding log-Sobolev setting) it is more convenient to work with $|\cdot|\dualnorm$ rather than with the equivalent norm $|\cdot|_{\theta^\ast_p}$ used in \cite{Nonlipschitz,ESAIM}.

In what follows we will need the following simple inequality which follows from convexity of $\theta$ and the assumption $\theta(0)=0$. For $x\in \RR^n$, $p > 0$, and $t \ge 1$,
\begin{align}\label{eq:norm-comparison}
|x|\dualnorm[tp] \le t|x|\dualnorm.
\end{align}

The following corollary to Proposition \ref{prop:Maurey} is again based on by now standard arguments, written however in the language of the norms $|\cdot|\dualnorm$.

\begin{corollary}\label{cor:concentration-Lipschitz} Let $X$ be a random vector with law $\mu$, satisfying  \eqref{eq:inf_conv_again} for all convex functions $f\colon \RR^n \to \RR$ bounded from below. Then for any smooth convex Lipschitz function $f\colon \RR^n \to \RR$ and $p \ge 0$,
\begin{align}\label{eq:concentration-Lipschitz}
\PP(|f(X) - \Med f(X)| > \sup_{x\in \RR^n}|\nabla f(x)|\dualnorm) \le 4e^{-p}.
\end{align}
\end{corollary}

\begin{remark}\label{rem:smooth-to-lipschitz}
It is easy to see that if the inequality \eqref{eq:concentration-Lipschitz} holds for all smooth convex Lipschitz functions, then one can apply it to arbitrary convex Lipschitz function, replacing $\sup_{x \in \RR^n} |\nabla f(x)|\dualnorm$ by the Lipschitz constant of $f$ with respect to the norm $|\cdot|_{\frac{1}{p}\theta}$. To verify this it is enough to consider convolutions of $f$ with a sequence of Gaussian densities converging to Dirac's mass at zero---they are smooth, have the same Lipschitz constant as $f$ and converge to $f$ uniformly (see e.g. \cite[p.~429]{MR1756011}).
\end{remark}

\begin{proof}[Proof of Corollary \ref{cor:concentration-Lipschitz}]
Let $A = \{y\in\RR^n\colon f(y) \le \Med f(X)\}$, so that $\PP(X\in A)\geq 1/2$. Then by convexity, for any $y \in A$,
\begin{align}\label{eq:upper-bound-median-gradient}
f(X) \le f(y) + \langle \nabla f(X),X-y\rangle \le \Med f(X) + |\nabla f(X)|\dualnorm \cdot |X-y |_{\frac{1}{p}\theta}.
\end{align}
Thus
\begin{align}
\PP(f(X) &> \Med f(X) + \sup_{x\in \RR^n}|\nabla f(x)|\dualnorm) \le \PP(\inf_{y \in A} |X-y |_{\frac{1}{p}\theta} > 1) \nonumber \\
&= \PP(X \notin A+ \cl B_{\theta}(p)) \le \frac{e^{-p}}{\PP(X \in A)} \le 2 e^{-p},\label{eq:upper-tail}
\end{align}
where in the second inequality we used Proposition \ref{prop:Maurey}.

Let now $A = \{y\in\RR^n \colon f(y) < \Med f(X) - \sup_{x\in \RR^n}|\nabla f(x)|\dualnorm\}$. Similarly as above, we obtain
\begin{align*}
1/2 &\le \PP( f(X) \ge \Med f(X)) \le \PP( \inf_{y\in A} |X-y|_{\frac{1}{p}\theta} \ge 1) \\
&\le \PP(X \notin  A+ B_{\theta}(p)) \le \frac{e^{-p}}{\PP(A)},
\end{align*}
which shows that
\begin{displaymath}
\PP(f(X) < \Med f(X) - \sup_{x\in \RR^N} |\nabla f(x)|_{\frac{1}{p}\theta}) \le 2e^{-p}.
\end{displaymath}
Combining the last inequality with \eqref{eq:upper-tail} proves the corollary.
\end{proof}

\subsection{Concentration inequalities for general convex functions}

We are now ready to state the main result of this section, contained in the following theorem, dealing with general (not necessarily Lipschitz) convex functions. In its formulation we adopt the convention $\frac{0}{0} = 0$. The proof of the theorem as well as of its corollary is postponed to Section \ref{sec:proof-nonlipschitz}

We would like to emphasize, that in the theorem we assume only \eqref{eq:concentration-Lipschitz}, which is streactly weaker than the infimum-convolution inequality \eqref{eq:inf_conv_again}.

\begin{theorem}\label{thm:concentration-nonlipschitz}
Let $X$ be a random vector satisfying  \eqref{eq:concentration-Lipschitz} for all smooth convex Lipschitz functions $f\colon \RR^n \to \RR$. Then for any smooth convex function $f \colon \RR^n \to \RR$, the following properties hold.
\begin{itemize}
\item[(i)] For any $p \ge 1$,
\begin{align}\label{eq:selfnormalized-moment}
\Big\|\frac{(f(X) - \Med f(X))_+}{|\nabla f(X)|\dualnorm}\Big\|_p \le 3^{1/p}.
\end{align}
\item[(ii)] Let $p > 0$, $q \in (1/2,1]$ and let $M_{p,q}\in \RR$ satisfy $\PP(|\nabla f(X)|\dualnorm \le M_{p,q}) \ge q$. Then
\begin{displaymath}
\PP\Big( f(X) < \Med f(X) - M_{p,q} \bigl(1+\log(8/(2q-1)) \bigr) \Big) \le 4e^{-p}.
\end{displaymath}

\noindent In particular for $p \ge 0$,
\begin{align}\label{eq:nonlipschitz-lower-tail}
\PP(f(X) < \Med f(X) - 16 \EE |\nabla f(X)|\dualnorm) \le 4e^{-p}.
\end{align}

\item[(iii)] For all $p > 0$,
\begin{displaymath}
\|(f- \Med f(X) )_-\|_p \le 48\EE |\nabla f(X)|\dualnorm.
\end{displaymath}
\end{itemize}
\end{theorem}

\begin{remark}\label{rem:wystarczy-upper-tail}
As will become clear in the proof, the part (i) of the above theorem holds in fact under one-sided concentration, i.e. it is enough to assume that
\begin{align}\label{eq:concentration-Lipschitz-1}
\PP(f(X) - \Med f(X) > \sup_{x\in \RR^n}|\nabla f(x)|\dualnorm) \le 4e^{-p}.
\end{align}

\end{remark}

Let us now illustrate the above theorem with a few concrete examples and a~corollary. In particular we will show what the norms $|\cdot|\dualnorm$ look like for different choices of the cost function $\theta$.

\begin{example}\label{ex:norms-and-concentration-1}
If $\theta(x) = c|x|^r$ for some $r \ge 1$ and $c > 0$, then $|x|\dualnorm = c^{-1/r}p^{1/r}|x|$ and \eqref{eq:concentration-Lipschitz} is equivalent to
\begin{align}\label{eq:simple-tail-condition}
\PP(|f(X) - \Med f(X)| \ge t) \le 4\exp(-c t^r)
\end{align}
for all 1-Lipschitz convex functions (in particular for $r=2$ we get the subgaussian concentration).
The first part of Theorem \ref{thm:concentration-nonlipschitz} gives then the following inequality for all (not necessarily Lipschitz) convex functions and $p\ge 1$,
\begin{displaymath}
  \Big\| \frac{(f(X)-\Med f(X))_+}{|\nabla f(X)|}\Big\|_p \le 3^{1/p} c^{-1/r} p^{1/r}.
\end{displaymath}
Thus by the $L^p$-Chebyshev inequality, with $p = ct^r/(3e)^r$ we obtain for $t\ge 0$,
\begin{align}\label{eq:first-example}
  \PP\Big(\frac{f(X)-\Med f(X)}{|\nabla f(X)|} \ge t\Big) \le e \exp\Big(-\frac{ct^r}{(3e)^r}\Big)
\end{align}
(the additional factor $e$ on the right-hand side is introduced artificially to encompass all $t\ge0$, also those for which $p < 1$; note that in this case the right-hand side exceeds one). We remark that similar self-normalized inequalities are known e.g. in the theory of empirical processes (see \cite{MR2073181}).

The lower tail inequalities gives
\begin{align}\label{eq:first-example-lower-tail}
  \PP(f(X) \le \Med f(X) - t) \le 4\exp\Big(-c\frac{t^r}{16^r (\EE|\nabla f(X)|)^r}\Big).
\end{align}
Moreover, using the full strength of part (ii) of Theorem \ref{thm:concentration-nonlipschitz}, one can replace $\EE|\nabla f(X)|$ by $4^{-1}M_{3/4}$, where $M_{3/4}$ is the $3/4$ quantile of $|\nabla f(X)|$.  Thus no integrability conditions on the gradient are in fact required.
\end{example}

\begin{remark}  Let us note that inequalities similar to~\eqref{eq:first-example-lower-tail} were previously known with the quantity $(\EE |\nabla f(X)|^2)^{1/2}$ instead of the quantile or $\EE |\nabla f(X)|$ (see \cite{MR1756011} or \cite[Chapter 3.3]{MR1767995}. Very recently, Paouris and Valettas \cite{2016arXiv161101723P} have proved that the standard Gaussian vector in $\RR^n$ satisfies a similar inequality (for $r=2$) with $\EE |f(X) - \Med f(X)|$ in place of $\EE |\nabla f(X)|$. Their proof uses in a crucial way isoperimetric properties of Gaussian measures. The version with $\EE |\nabla f(X)|$ follows simply by an application of the (1,1)-Poincar\'e inequality for the Gaussian measure, i.e. $\EE |f(X) - \Med f(X)| \le C \EE |\nabla f(X)|$ (see e.g. \cite{MR864714,MR2507637}). In fact the proof in \cite{2016arXiv161101723P} gives also inequalities in terms of quantiles of $|f(X) - M|$. We do not know if they are comparable to our estimates (specialized to the standard Gaussian measure) in terms of quantiles of $|\nabla f(X)|$.
\end{remark}

Note also that \eqref{eq:simple-tail-condition} for $r=1$ is a consequence of the convex Poincar\'e inequality (however we do not know if \eqref{eq:convex-Poincare} implies \eqref{eq:simple-tail-condition} with $c$ depending only on $\lambda$ and not on the dimension $n$, see Question \ref{q:lower-tail} below).

\begin{example}\label{ex:norms-and-concentration-2}
Let us now consider a measure $\mu$ on $\RR^n$ satisfying the convex Poincar\'e inequality with constant $\lambda$. Then, by Theorem \ref{thm:BL-convex} it satisfies the convex Bobkov-Ledoux inequality \eqref{eq:Bobkov-Ledoux-inequality-convex} with constants $C$ and $c$ depending only on $\lambda$. By the classical Herbst argument it follows  (see e.g. \cite{MR1440138,MR3456588}) that for each $N \ge 1$, if $X$ is an $Nn$-dimensional random vector with law $\mu^{\otimes N}$, then for any smooth convex function $f\colon \RR^{Nn} \to \RR$ and any $t > 0$,
\begin{align*}
&\PP( f(X) \ge \EE f(X) + t) \\
&\le 2\exp\bigl(-c'(\lambda)\min\Bigl\{\frac{t^2}{\sup_{x\in \RR^{Nn}} |\nabla f(x)|^2},\frac{t}{\sup_{x\in \RR^n} \max_{i\le N} |\nabla_i f(x)|}\Bigr\}\bigr),
\end{align*}
where for $x = (x_1,\ldots,x_N) \in (\RR^n)^N =  \RR^{Nn}$, $\nabla_i f(x)$ denotes the partial gradient with respect to $x_i$.

Moreover, by the Poincar\'e inequality
\begin{equation*}
|\EE f(X) - \Med f(X)| \le 1/\sqrt{\lambda} \sup_{x\in \RR^{Nn}} |\nabla f(x)|,
\end{equation*}
 which at the cost of changing the constant allows to replace the mean by the median in the above inequality.
Thus we obtain that for some constant $c''(\lambda)$ and $p > 0$,
\begin{displaymath}
\PP\bigl( f(X) \ge \Med f(X) + c''(\lambda) \sup_{x\in \RR^{Nn}} (\sqrt{p}|\nabla f(x)| + p\max_{i\le N} |\nabla_i f(x)|)\bigr) \le 2 e^{-p}.
\end{displaymath}
It is easy to see that up to universal constants $c''(\lambda) (\sqrt{p}|x| +p\max_{i\le N} |x_i|)$ is equivalent to $|x|\dualnorm$, where
\begin{displaymath}
\theta(x) = \sum_{i=1}^N \min\Bigl\{\Big|\frac{x_i}{c''(\lambda)}\Big|^2,\Big|\frac{x_i}{c''(\lambda)}\Big|\Bigr\}.
\end{displaymath}
More precisely
\begin{displaymath}
|x|\dualnorm \le c''(\lambda) \bigl(\sqrt{p}|x| +p\max_{i\le N} |x_i|\bigr) \le 2|x|\dualnorm.
\end{displaymath}

Thus, the first part of Theorem~\ref{thm:concentration-nonlipschitz} together with Remark~\ref{rem:wystarczy-upper-tail} gives for arbitrary smooth convex function $f$ on $\RR^{Nn}$, the inequality
    \begin{displaymath}
      \Big\|\frac{(f(X)-\Med f(X))_+}{\sqrt{p}|\nabla f(X)| + p \max_{i\le N}|\nabla_i f(X)|}\Big\|_p \le c'''(\lambda),
    \end{displaymath}
for $p\ge 1$, where $c'''(\lambda)$ depends only on $\lambda$. By Chebyshev's inequality this implies that
\begin{displaymath}
  \PP\Big(\frac{(f(X)-\Med f(X))_+}{\sqrt{t}|\nabla f(X)| + t \max_{i\le N}|\nabla_i f(X)|} \ge ec'''(\lambda)\Big) \le e^{-t}
\end{displaymath}
for $t\ge 1$ (note that contrary to \eqref{eq:first-example} this time $t$ cannot be removed from the denominator).

As for the lower tail, by Theorem \ref{thm:main}, Remark \ref{rem:dependence-of-constants}, Lemma \ref{le:inf-convolution} and tensorization properties of infimum convolution inequalities (see Lemma 5 in \cite{MR1097258}) we obtain that $X$ satisfies \eqref{eq:inf_conv_again} and thus also \eqref{eq:concentration-Lipschitz} with $\theta(x) = K(\lambda,n) \sum_{i=1}^N \min(|x_i|^2,|x_i|)$, where $K(\lambda,n)$ depends only on $\lambda$ and the dimension $n$. Thus, by the second part of Theorem \ref{thm:concentration-nonlipschitz},
\begin{displaymath}
\PP(f(X) \le \Med f(X) - K'(\lambda,n)\bigl[\sqrt{p}\EE|\nabla f(X)| + p \EE\max_{i\le N}|\nabla_i f(X)|\bigr] ) \le 4e^{-p},
\end{displaymath}
or equivalently (up to constants depending only on $\lambda,n$),
\begin{multline*}
  \PP(f(X) \le \Med f(X) - t) \\
  \le 4\exp\Big( -K''(\lambda,n) \min \Bigl\{ \frac{t^2}{(\EE|\nabla f(X)|)^2}, \frac{t}{\EE\max_{i\le N}|\nabla_i f(X)|} \Bigr\}\Big).
\end{multline*}

We stress that all the above inequalities are dimension-free in the sense that the constants do not depend on the number $N$ but just on the initial dimension $n$ (cf. Remark \ref{re:1.5}).
\end{example}

\begin{example}\label{ex:norms-and-concentration-3}
 Finally, we remark that general cost functions $\theta$ lead to other concentration profiles, which have been studied in the literature. One can for instance consider products of measures on $\RR$, satisfying \eqref{eq:inf_conv_again} with
 \begin{equation*}
 \theta(x) = c(|x|^2\ind{|x|\le 1} + |x|^r\ind{|x|>1})
 \end{equation*}
 for $r \ge 1$ (such measures are characterized thanks to results in \cite{gozlan_new}). If we denote for $x \in \RR^n$, $|x|_r = (|x_1|^r+\cdots+|x_n|^r)^{1/r}$ and let $r^\ast$ be the H\"older conjugate of $r$, then  such costs correspond for $r \in [1,2]$ to norms of the form $|x|\dualnorm \simeq \sqrt{p}|x| + p^{1/r}|x|_{r^\ast}$ (the case $r=1$ has been discussed above), while for $r > 2$ to
 \begin{equation*}
 |x|\dualnorm\simeq p^{1/r}|(x^\ast_i)_{i=1}^p|_{r^\ast} + \sqrt{p}|(x_i^\ast)_{i=p+1}^n|,
 \end{equation*}
 where $(x_i^\ast)_{i=1}^n$ is the non-increasing rearrangement of the sequence $(|x_i|)_{i=1}^n$.
\end{example}

We will now present a corollary to Theorem \ref{thm:concentration-nonlipschitz}, providing concentration inequalities for non-Lipschitz convex functions, in the spirit of recent results due to Bobkov, Nayar, and Tetali~\cite{BobkovNayarTetali}.

\begin{corollary}\label{cor:nonlipschitz} Under the assumptions of Theorem \ref{thm:concentration-nonlipschitz} for all convex functions $f\colon \RR^n \to \RR$,
\begin{displaymath}
\PP(f(X)-\Med f(X) \ge t) \le \inf_{p\ge 1} \bigl\{ e^{-p} +  \PP\bigl(|\nabla f(X)|\dualnorm \ge t/(3e)\bigr)\bigr\}.
\end{displaymath}
Moreover, for any $p\ge 1$,
\begin{align}\label{eq:moments-quantiles}
  \PP\big(|f(X) - \Med f(X)| \ge 3e^2\big\| |\nabla f(X)|\dualnorm \big\|_p\big) \le 6e^{-p}
\end{align}
\end{corollary}

Let us  note that inequalities of the form \eqref{eq:moments-quantiles} have been obtained in \cite{ESAIM} for all smooth functions of random vectors satisfying modified log-Sobolev inequalities (assumed to hold for all smooth functions). Therein, the function $\theta$ had to satisfy some appropriate growth condition.

\begin{example}
In particular for $\theta(x) = c|x|^2$, the above corollary gives
\begin{displaymath}
  \PP(f(X)-\Med f(X) \ge t) \le \inf_{p\ge1 } \bigl\{e^{-p} +  \PP(\sqrt{p/c}|\nabla f(X)| \ge t/(3e))\bigr\}.
\end{displaymath}
By substituting $p = \frac{c t^2}{(3 e)^2 L^2}$ and adjusting the constant we obtain
\begin{align}\label{eq:subgaussian-nonlipschitz-upper-tail}
  \PP(f(X)-\Med f(X) \ge t) \le \inf_{L>0} \bigl\{ 2e^{-c'\frac{t^2}{L^2}} +  \PP(|\nabla f(X)| \ge L)\bigr\},
\end{align}
where $c'$ is positive and depends only on $c$. The factor 2 in the above inequality is introduced for notational simplicity to allow the whole range of $L > 0$ in the infimum (note that for large $L$ we have $p < 1$ and we cannot apply Corollary \ref{cor:nonlipschitz}, on the other hand the above inequality becomes then trivial, as the right-hand side exceeds one).

Recall also the second part of Theorem \ref{thm:concentration-nonlipschitz} which for $q = 3/4$ gives in this case
\begin{align}\label{eq:subgaussian-nonlipschitz-lower-tail}
  \PP(f(X) \le \Med f(X) - t) \le 4\exp\bigl(-c''\frac{t^2}{M_{3/4}^2}\bigr),
\end{align}
where $M_{3/4} = \inf\{x\in\RR^n \colon \PP(|\nabla f(X)| \le x) \ge 3/4\}$ and $c''$ again depends only on~$c$.
\end{example}

The above inequalities should be compared with a recent result in \cite{BobkovNayarTetali}, which asserts that for some constant positive $c'''$ depending only on $c$,
\begin{align}\label{eq:BNT}
\PP(|f(X) - f(Y)| \ge t) \le 2 \inf_{L \ge \Med |\nabla f(X)|} \bigl\{ e^{-c'''\frac{t^2}{L^2}} + \PP(|\nabla f(X)| \ge L) \bigr\},
\end{align}
where $Y$ is an independent copy of $X$.

It is not difficult to see  that in the regime of $t$ for which the above inequalities are of interest, i.e. the right-hand sides are small, \eqref{eq:subgaussian-nonlipschitz-upper-tail} gives estimates on the upper tail which (up to numerical constants) are comparable to those implied by \eqref{eq:BNT}, whereas for the lower tail, the inequality \eqref{eq:subgaussian-nonlipschitz-lower-tail} improves over \eqref{eq:BNT}.

\begin{example}
Consider now $\theta(x) = \sum_{i=1}^N \min(|x_i/c|^2,|x_i/c|)$, which we have already discussed in Example~\ref{ex:norms-and-concentration-2}.
From Corollary \ref{cor:nonlipschitz} we get
\begin{displaymath}
\PP(f(X)-\Med f(X) \ge t) \le \inf_{p\ge 1} \bigl\{ e^{-p} +  \PP(\sqrt{p}|\nabla f(X)|+ p\max_{i\le N}|\nabla_i f(X)|\ge t/c')\bigr\}.
\end{displaymath}
By substituting $p = \min\{\frac{t^2}{(2c')^2L^2},\frac{t}{2c'M}\}$ and using the union bound we obtain
\begin{multline*}
\PP(f(X)-\Med f(X) \ge t) \\
\le\inf_{L,M>0} \Big\{ 2\exp\Big(-c''\min\bigl\{\frac{t^2}{L^2},\frac{t}{M}\bigr\}\Big) +  \PP\bigl(|\nabla f(X)| \ge L\bigr)\\
+\PP\bigl(\max_{i\le N}|\nabla_i f(X)|\ge M\bigr)\Big\},
\end{multline*}
with $c''$ depending only on $c$. As in the preceding example, the factor 2 is introduced to allow for all positive values of $L,M$.
\end{example}

\begin{remark}
Let us note that another way of obtaining estimates on the upper tail of non-Lipschitz functions under the convex Poincar\'e inequality is to use the estimates  \eqref{eq:Poincare-moments} and \eqref{eq:Poincare-moments-tail}. By approximating arbitrary convex functions with Lipschitz ones we can easily see that they hold in fact for all convex functions. Thus, if one controls the moments of $|\nabla f(X)|$, one can obtain tail estimates beyond the Lipschitz case. Such inequalities are however different than those of the above example as they are of exponential type and not of mixed exponential or Gaussian type. On the other hand, the weak transportation inequality with cost $\theta(x) = c\sum_{i=1}^n \min(|x_i|^2,|x_i|)$ arises usually as a consequence of tensorization, so in order to apply it we need some additional product structure of the measure.
\end{remark}

\subsection{Proofs of Theorem \ref{thm:concentration-nonlipschitz} and Corollary \ref{cor:nonlipschitz}}\label{sec:proof-nonlipschitz}

\begin{proof}[Proof of Theorem \ref{thm:concentration-nonlipschitz}]
Let us start with (i), the proof of which is quite similar to the proof of Corollary \ref{cor:concentration-Lipschitz}.
Let us again define $A = \{x\in\RR^n\colon f(x) \le \Med f(X)\}$.
Using \eqref{eq:norm-comparison} and \eqref{eq:upper-bound-median-gradient}, we can write for $t \ge 1$,
\begin{displaymath}
\frac{f(X) - \Med f(X)}{t |\nabla f(X)|\dualnorm} \le \frac{f(X) - \Med f(X)}{|\nabla f(X)|\dualnorm[tp]} \le \inf_{y\in A}|X - y|_{\frac{1}{tp}\theta}.
\end{displaymath}
Hence for $t \ge 1$,
\begin{displaymath}
\PP\Big( \frac{f(X) - \Med f(X)}{|\nabla f(X)|\dualnorm }  > t \Big) \le \PP(\inf_{y\in A} |X - y|_{\frac{1}{tp}\theta} > 1) \le 4 e^{-pt},
\end{displaymath}
where we used the fact that the function $g(x) = \inf_{y\in A} |x - y|_{\frac{1}{tp}\theta}$ is convex, 1-Lipschitz with respect to $|\cdot|_{\frac{1}{tp}\theta}$ and $\Med g(X)=0$, together with Corollary \ref{cor:concentration-Lipschitz} and Remark \ref{rem:smooth-to-lipschitz}.
We can now integrate by parts and get
\begin{displaymath}
\EE \Big| \frac{(f(X) - \Med f(X))_+}{|\nabla f(X)|\dualnorm}\Big|^p \le 1 + 4\int_1^\infty pt^{p-1}e^{-pt}dt \leq 1 + 4\int_1^\infty e^{-t}dt \leq 3
\end{displaymath}
(the integrand is pointwise non-increasing with respect to $p \ge 1$, as the computation of the derivative with respect to $p$ reveals),  which proves the first part of the theorem.

Let us now pass to the second part. Assume without loss of generality that $\Med f(X) = 0$. Consider the set  $B = \{x\in\RR^n\colon |\nabla f(x)|\dualnorm\le M_{p,q}\}$. By the definition of $M_{p,q}$, we have $\PP(X \in B) \ge q$. Let $\tilde{f}\colon \RR^n \to \RR$ be defined as
\begin{displaymath}
\tilde{f}(x) = \sup_{y \in B} \bigl\{ f(y) + \langle \nabla f(y),x-y\rangle \bigr\}.
\end{displaymath}
Then $\tilde{f}$ is convex, moreover by convexity of $f$ we have $\tilde{f} \le f$ pointwise and $\tilde{f} = f$ on $B$. By the definition of the set $B$ and inequality~\eqref{eq:norm-comparison}, for any $t\ge 1$ all linear functionals $x \mapsto \langle \nabla f(y),x\rangle$, $y \in B$, are $(tM_{p,q})$-Lipschitz with respect to $|\cdot|_{\frac{1}{tp}\theta}$ and therefore so is $\tilde{f}$. By Corollary \ref{cor:concentration-Lipschitz} and Remark \ref{rem:smooth-to-lipschitz} this implies that for any $t \ge 1$,
\begin{align}\label{eq:concentration-modified-function}
\PP(|\tilde{f}(X) - \Med \tilde{f}(X)| > tM_{p,q}) \le 4 e^{-tp}.
\end{align}
We also have $\PP(\tilde{f}(X) \ge 0) \ge \PP(f(X) \ge 0 \;\textrm{and}\; X \in B) \ge q - 1/2$. Therefore, the above inequality applied with $t \searrow \log(8/(2q-1))> 1$ gives
\begin{displaymath}
\Med \tilde{f}(X) + M_{p,q}  \log(8/(2q-1)) \ge 0,
\end{displaymath}
 which by another application of \eqref{eq:concentration-modified-function} implies
 \begin{displaymath}
 \PP\Big(f(X) < - M_{p,q} \bigl(1+ \log(8/(2q-1))\bigr)\Big)\le \PP(\tilde{f}(X) < \Med \tilde{f}(X) - M_{p,q}) \le 4e^{-p}.
 \end{displaymath}
This proves the first inequality of part (ii).

The second inequality of part (ii) follows from the first one by specializing to $q = 3/4$, $M_{p,q} = 4\EE |\nabla f(X)|\dualnorm$ and some elementary calculations.

As for part (iii), using again \eqref{eq:norm-comparison} and \eqref{eq:nonlipschitz-lower-tail}, we get for $t \ge 16\EE |\nabla f(X)|\dualnorm$
\begin{displaymath}
\PP( f(X) - \Med f(X) \le - t) \le 4\exp\Big(-\frac{pt}{16\EE |\nabla f(X)|\dualnorm}\Big).
\end{displaymath}

Now, again by integration by parts,
\begin{align*}
  \EE &(f(X) - \Med f(X))_-^p \\
  &\le (16\EE |\nabla f(X)|\dualnorm)^p + 4p\int_{16\EE |\nabla f(X)|\dualnorm}^\infty t^{p-1}\exp\Big(-\frac{pt}{16\EE |\nabla f(X)|\dualnorm}\Big)dt\\
  & \le 3(16 \EE |\nabla f(X)|\dualnorm)^p,
  \end{align*}
  which ends the proof.
\end{proof}

\begin{proof}[Proof of Corollary \ref{cor:nonlipschitz}]
To prove the first inequality it is enough to note that if $|\nabla f(X)|\dualnorm \le t/(3e)$ and $f(X) - \Med f(X) \ge t$, then
\begin{displaymath}
  Z:= \frac{(f(X)-\Med f(X))_+}{|\nabla f(X)|\dualnorm} \ge 3 e \ge e\|Z\|_p,
\end{displaymath}
where the last inequality follows from \eqref{eq:selfnormalized-moment}. The assertion follows thus from Chebyshev's inequality: $\PP(Z \ge e\|Z\|_p) \le e^{-p}$.

As for the second inequality, we apply the first one with $t = 3e^2 \||\nabla f(X)|\dualnorm\|_p$ and combine it with the estimate \eqref{eq:nonlipschitz-lower-tail}.
\end{proof}

\section{Further questions}\label{sec:final}

Let us conclude with some open questions, which seem natural in view of our results.

As already mentioned in the introduction, in our proof of the implication
\begin{align*}
&\mu \; \text{satisfies the convex Poincar\'e inequality with constant}\; \lambda \\
&\implies  \mu \; \text{satisfies  the inequality }\; \Tbar_{\theta_{C,D}}\;  \text{for some} \; C,D,
\end{align*}
the constants $C,D$ do not depend just on $\lambda$, but also on certain quantiles of the measure $\mu$. In fact, the issue comes from the inequality $\Tbar^+$, since the constants in $\Tbar^-$ do depend only on $\lambda$ (see~Corollary~\ref{cor:Poincare-transport-minus}). This gives rise to our first question.

\begin{question} Does the Poincar\'e inequality with constant $\lambda$ imply the weak transportation inequality $\Tbar_{\theta_{C,D}}$ with constants $C,D$ depending only on $\lambda$?
\end{question}

The inspection of our proof shows that in order to answer the above question in the affirmative, it is enough to remove the restriction on $t$ in Lemma \ref{le:conc-Med-concave-n}.
An improved version of this lemma, valid for all $t>0$ would follow by part (ii) of Theorem \ref{thm:concentration-nonlipschitz} provided that one can show that the convex Poincar\'e inequality with constant $\lambda$ implies subexponential concentration for convex 1-Lipschitz functions, with constants depending only on $\lambda$. The problem lies in the lower-tail (as the upper one is handled by Proposition \ref{prop:Poincare-upper-tail}). More precisely, we have the following result.

\begin{theorem} Assume that $\mu$ is a probability measure on $\RR^n$, satisfying the convex Poincar\'e inequality \eqref{eq:convex-Poincare} with constant $\lambda$ and $c$ is a positive constant, such that for all $1$-Lipschitz convex functions $f\colon \RR^n\to \RR$ and all $t > 0$,
\begin{displaymath}
\mu\bigl(\{x\in\RR^n\colon f(x) \le \Med_\mu f  - t\}\bigr) \le 2 \exp(-c t).
\end{displaymath}
Then $\mu$ satisfies the inequality $\Tbar_{\theta_{C,D}}$ with $C, D$ depending only on $\lambda$ and $c$.
\end{theorem}

This motivates the following question, which is clearly of interest also in its own right.

\begin{question}\label{q:lower-tail} Does the convex Poincar\'e inequality \eqref{eq:convex-Poincare} with constant $\lambda$ imply subexponential estimates for the lower-tail of convex 1-Lipschitz functions, with constants depending only on $\lambda$? Specifically, is it true that whenever $\mu$ is a probability measure on $\RR^n$ satisfying \eqref{eq:convex-Poincare}, then for every convex $1$-Lipschitz function $f\colon \RR^{n} \to \RR$,
\begin{displaymath}
\mu\bigl(\{x\in\RR^n\colon f(x) \le \Med_\mu f  - t\}\bigr) \le 2\exp(-c(\lambda) t),
\end{displaymath}
where the constant $c(\lambda)$ depends only on $\lambda$?
\end{question}

The inequality provided by Lemma \ref{le:conc-Med-concave-n} introduces an additional dependence on $n$, which carries over to the dependence of constants in Theorem \ref{thm:main}. Let us point out that all the proofs of lower-tail estimates based on the Poincar\'e inequality and available for the category of all smooth functions, which we have been able to find in the literature, seem to break down in the convex setting (see e.g. the arguments in \cite{MR708367,MR1258492,MR3311918}).

\appendix

\section{Facts related to Hamilton-Jacobi equations}\label{sec:HJ}

We will now present some basic properties of Hamilton-Jacobi equations related to infimum convolution operators with the cost $\theta(x) = \alpha(|x|)$, where $\alpha$ is given by \eqref{eq:cost-alpha}, which have been exploited in the proof of Lemma \ref{lem:logSob-transport}. We remark that all the facts we will rely on are quite standard, however in the literature they are usually considered under slightly different sets of assumptions, which makes it difficult to find an off the shelf result applicable to our situation. We will briefly indicate how the reasonings from \cite[Chapter 3]{evans} can be modified to yield the properties we need. Alternatively, as in \cite{gozlan_new}, one could rely on modification of the results from \cite{MR3186934}, where the theory of Hamilton-Jacobi equations is extended to the setting of metric spaces.

\begin{proposition} \label{prop:HJ}Let $C,L$ be positive constants and let $\alpha$ be defined by \eqref{eq:cost-alpha}. Assume that $f\colon \RR^n \to \RR$ is either bounded from below or $L$-Lipschitz and let $u\colon (0,\infty)\times \RR^n \to \RR$ be given by
$u(t,x) = Q_t^\alpha f(x)$, where
\begin{equation*}
Q_t^{\alpha} f(x) = \inf_{y\in\RR^n} \{ f(y) + t\alpha(|x-y|/t) \},\quad t>0.
\end{equation*}
Then the following conditions hold.
\begin{itemize}
\item[(a)] For every $s,t > 0$ and every $x \in \RR^n$, $Q_tQ_s f(x) = Q_{t+s} f(x)$.

\item[(b)] The function $u$ is Lipschitz on $(0,\infty)\times \RR^n$,

\item[(c)] At every point $(t,x)\in(0,\infty)\times\RR^n$ of differentiability of $u$, one has
\begin{displaymath}
\frac{d}{dt} u(t,x) + \alpha^\ast(|\nabla_x u(t,x)|) = 0,
\end{displaymath}
where $\alpha^\ast$ is the Legendre transform of $\alpha$, given explicitly by the formula
\begin{equation*}
\alpha^*(s) =\begin{cases}
C|s|^2 & \text{for } |s| \le L,\\
+\infty  & \text{for } |s|>L.
\end{cases}
\end{equation*}
\end{itemize}
\end{proposition}

\begin{proof}[Sketch of proof] Let us note that if $f$ is bounded from below or $L$-Lipschitz, then $Q_t f$ is well defined.

\emph{Ad (a).} To show the semigroup property one can repeat the argument from the proof of \cite[Chapter 3.3.2, Lemma 1]{evans}, however in our setting one needs to work with infima rather then minima.

\emph{Ad (b).} For fixed $t$, $u$ is $L$-Lipschitz as the function of $x$, as an infimum of $L$-Lipschitz functions. Indeed for each $y$, the function $x \mapsto t\alpha(|x-y|/t)$ is $L$-Lipschitz. As for the Lipschitz property with respect to $t$, the argument in the proof of \cite[Chapter 3.3.2, Lemma 2]{evans} shows that if $f$ is $L$-Lipschitz, then for any $x$,
\begin{displaymath}
|u(t,x) - f(x)| \le Mt,
\end{displaymath}
where $M = \max_{|x| \le L} \alpha^\ast(x) = CL^2$. Now the Lipschitz condition with respect to $t>0$ (for general $f$, which may not be $L$-Lipschitz) follows from the semigroup property and the fact that $Q_t f$ is an $L$-Lipschitz function of $x$.

\emph{Ad (c).}  Using again the fact that $Q_t f$ is $L$-Lipschitz, it is enough to consider the case when so is $f$. One can then repeat the proof of \cite[Chapter 3.3.2, Theorem 5]{evans}, provided that one can prove that the infimum in the definition of $Q_t f$ is in fact achieved. To this end, it is enough to note that whenever $|y-x| > 2CLt$ we have, denoting $z = x+ 2CLt(y-x)/|x-y|$,
\begin{align*}
f(y&)+t\alpha(|x-y|/t) \\
&= f(z) +  t\alpha(|x-z|/t) + (f(y) - f(z)) + t\alpha(|x-y|/t) - t\alpha(|x-z|/t)\\
&\ge
f(z)  + t\alpha(|x-z|/t) - L|z-y| + t\alpha(|x-y|/t) - t\alpha(|x-z|/t)\\
&=f(z)  + t\alpha(|x-z|/t),
\end{align*}
where the inequality holds by the Lipschitz property of $f$ and the last equality follows from the definition of $\alpha$ (and the fact that $z$ lies on the interval with endpoints $x$ and $y$).
Thus $Q_t f(x) = \inf_{|y-x| \le 2CL} \{f(y) + t\alpha(|y-x|/t)\}$ and the existence of the minimizer follows from compactness and continuity of $f$ and $\alpha$.
\end{proof}

\bibliographystyle{amsplain}	
\bibliography{ConvexPoincareTransportation}

\providecommand{\bysame}{\leavevmode\hbox to3em{\hrulefill}\thinspace}
\providecommand{\MR}{\relax\ifhmode\unskip\space\fi MR }
\providecommand{\MRhref}[2]{%
  \href{http://www.ams.org/mathscinet-getitem?mr=#1}{#2}
}
\providecommand{\href}[2]{#2}
\begin{thebibliography}{10}

\bibitem{ESAIM}
{R}ados{\l}aw Adamczak, {W}itold Bednorz, and {P}awe{\l} Wolff, \emph{Moment
  estimates implied by modified log-{S}obolev inequalities}, to appear in
  ESAIM: Probability and Statistics.

\bibitem{MR3456588}
Rados{\l}aw Adamczak and Micha{\l} Strzelecki, \emph{Modified log-{S}obolev
  inequalities for convex functions on the real line. {S}ufficient conditions},
  Studia Math. \textbf{230} (2015), no.~1, 59--93. \MR{3456588}

\bibitem{Nonlipschitz}
Rados{\l}aw Adamczak and Pawe{\l} Wolff, \emph{Concentration inequalities for
  non-{L}ipschitz functions with bounded derivatives of higher order}, Probab.
  Theory Related Fields \textbf{162} (2015), no.~3-4, 531--586. \MR{3383337}

\bibitem{MR1258492}
S.~Aida and D.~Stroock, \emph{Moment estimates derived from {P}oincar\'e and
  logarithmic {S}obolev inequalities}, Math. Res. Lett. \textbf{1} (1994),
  no.~1, 75--86. \MR{1258492}

\bibitem{MR1845806}
C\'ecile An\'e, S\'ebastien Blach\`ere, Djalil Chafa\"\i, Pierre Foug\`eres,
  Ivan Gentil, Florent Malrieu, Cyril Roberto, and Gr\'egory Scheffer,
  \emph{Sur les in\'egalit\'es de {S}obolev logarithmiques}, Panoramas et
  Synth\`eses [Panoramas and Syntheses], vol.~10, Soci\'et\'e Math\'ematique de
  France, Paris, 2000, With a preface by Dominique Bakry and Michel Ledoux.
  \MR{1845806}

\bibitem{MR1440138}
S.~Bobkov and M.~Ledoux, \emph{Poincar\'e's inequalities and {T}alagrand's
  concentration phenomenon for the exponential distribution}, Probab. Theory
  Related Fields \textbf{107} (1997), no.~3, 383--400. \MR{1440138}

\bibitem{MR1682772}
S.~G. Bobkov and F.~G{\"o}tze, \emph{Exponential integrability and
  transportation cost related to logarithmic {S}obolev inequalities}, J. Funct.
  Anal. \textbf{163} (1999), no.~1, 1--28. \MR{1682772}

\bibitem{BobkovNayarTetali}
Sergey {Bobkov}, Piotr {Nayar}, and Prasad {Tetali}, \emph{{Concentration
  Properties of Restricted Measures with Applications to Non-Lipschitz
  Functions}}, To appear in GAFA Seminar Notes (2015),
  \texttt{arXiv:1506.06174}.

\bibitem{MR1846020}
Sergey~G. Bobkov, Ivan Gentil, and Michel Ledoux, \emph{Hypercontractivity of
  {H}amilton-{J}acobi equations}, J. Math. Pures Appl. (9) \textbf{80} (2001),
  no.~7, 669--696. \MR{1846020}

\bibitem{bobkov-goetze}
{S.G.} Bobkov and {F.} G{\"o}tze, \emph{Discrete isoperimetric and
  {P}oincar{\'e}-type inequalities}, Probab. Theory Relat. Fields \textbf{114}
  (1999), no.~2, 245--277.

\bibitem{MR2641771}
Djalil Chafa{\"\i} and Florent Malrieu, \emph{On fine properties of mixtures
  with respect to concentration of measure and {S}obolev type inequalities},
  Ann. Inst. Henri Poincar\'e Probab. Stat. \textbf{46} (2010), no.~1, 72--96.
  \MR{2641771}

\bibitem{MR2073181}
Victor~H. de~la Pe\~na, Michael~J. Klass, and Tze~Leung Lai,
  \emph{Self-normalized processes: exponential inequalities, moment bounds and
  iterated logarithm laws}, Ann. Probab. \textbf{32} (2004), no.~3A,
  1902--1933. \MR{2073181}

\bibitem{evans}
{L. C.} Evans, \emph{Partial differential equations}, second ed., Graduate
  Studies in Mathematics, vol.~19, American Mathematical Society, Providence,
  RI, 2010.

\bibitem{Feldheim}
Naomi Feldheim, Arnaud Marsiglietti, Piotr Nayar, and Jing Wang, \emph{A note
  on the convex infimum convolution inequality}, to appear in Bernoulli,
  preprint (2015), \texttt{arXiv:1505.00240}.

\bibitem{MR1338834}
E.~D. Gluskin and S.~Kwapie\'n, \emph{Tail and moment estimates for sums of
  independent random variables with logarithmically concave tails}, Studia
  Math. \textbf{114} (1995), no.~3, 303--309. \MR{1338834}

\bibitem{gozlan}
{N.} Gozlan, {C.} Roberto, {P.M.} Samson, and {P.} Tetali, \emph{Kantorovich
  duality for general transport costs and applications}, to appear in J. Funct.
  Anal., preprint (2014), {\tt arXiv:1412.7480v4}.

\bibitem{gozlan_new}
{N.} Gozlan, {C.} Roberto, {Y} Samson, {P.M.}~Shu, and {P.} Tetali,
  \emph{Characterization of a class of weak transport-entropy inequalities on
  the line}, to appear in Ann. Inst. Henri Poincar{\'e} Probab. Stat., preprint
  (2015), {\tt arXiv:1509.04202v2}.

\bibitem{MR3186934}
Nathael Gozlan, Cyril Roberto, and Paul-Marie Samson, \emph{Hamilton {J}acobi
  equations on metric spaces and transport entropy inequalities}, Rev. Mat.
  Iberoam. \textbf{30} (2014), no.~1, 133--163. \MR{3186934}

\bibitem{MR3311918}
\bysame, \emph{From dimension free concentration to the {P}oincar\'e
  inequality}, Calc. Var. Partial Differential Equations \textbf{52} (2015),
  no.~3-4, 899--925. \MR{3311918}

\bibitem{MR708367}
M.~Gromov and V.~D.~and Milman, \emph{A topological application of the
  isoperimetric inequality}, Amer. J. Math. \textbf{105} (1983), no.~4,
  843--854. \MR{708367}

\bibitem{MR1399224}
M.~Ledoux, \emph{On {T}alagrand's deviation inequalities for product measures},
  ESAIM Probab. Statist. \textbf{1} (1995/97), 63--87 (electronic). \MR{1399224
  (97j:60005)}

\bibitem{MR1849347}
\bysame, \emph{The concentration of measure phenomenon}, Mathematical Surveys
  and Monographs, vol.~89, American Mathematical Society, Providence, RI, 2001.
  \MR{1849347 (2003k:28019)}

\bibitem{MR1767995}
Michel Ledoux, \emph{Concentration of measure and logarithmic {S}obolev
  inequalities}, S\'eminaire de {P}robabilit\'es, {XXXIII}, Lecture Notes in
  Math., vol. 1709, Springer, Berlin, 1999, pp.~120--216. \MR{1767995}

\bibitem{MR1097258}
B.~Maurey, \emph{Some deviation inequalities}, Geom. Funct. Anal. \textbf{1}
  (1991), no.~2, 188--197. \MR{1097258 (92g:60024)}

\bibitem{MR2507637}
Emanuel Milman, \emph{On the role of convexity in isoperimetry, spectral gap
  and concentration}, Invent. Math. \textbf{177} (2009), no.~1, 1--43.
  \MR{2507637}

\bibitem{2016arXiv161101723P}
Grigoris {Paouris} and Petros {Valettas}, \emph{{A small deviation inequality
  for convex functions}}, preprint (2016), \texttt{arXiv:1611.01723}.

\bibitem{MR864714}
Gilles Pisier, \emph{Probabilistic methods in the geometry of {B}anach spaces},
  Probability and analysis ({V}arenna, 1985), Lecture Notes in Math., vol.
  1206, Springer, Berlin, 1986, pp.~167--241. \MR{864714}

\bibitem{MR1756011}
P.-M. Samson, \emph{Concentration of measure inequalities for {M}arkov chains
  and {$\Phi$}-mixing processes}, Ann. Probab. \textbf{28} (2000), no.~1,
  416--461. \MR{1756011 (2001d:60015)}

\bibitem{MR2073425}
\bysame, \emph{Concentration inequalities for convex functions on product
  spaces}, Stochastic inequalities and applications, Progr. Probab., vol.~56,
  Birkh\"auser, Basel, 2003, pp.~33--52. \MR{2073425 (2005d:60035)}

\bibitem{MR1122615}
M.~Talagrand, \emph{A new isoperimetric inequality and the concentration of
  measure phenomenon}, Geometric aspects of functional analysis (1989--90),
  Lecture Notes in Math., vol. 1469, Springer, Berlin, 1991, pp.~94--124.
  \MR{1122615 (93d:60095)}

\bibitem{TalConcMeasure}
\bysame, \emph{Concentration of measure and isoperimetric inequalities in
  product spaces}, Inst. Hautes \'Etudes Sci. Publ. Math. (1995), no.~81,
  73--205. \MR{1361756 (97h:60016)}

\end{thebibliography}

\end{document}